\documentclass[12pt]{article}
\usepackage [dvips]{graphics}
\usepackage[centertags]{amsmath}
\usepackage{amsfonts}
\usepackage{amssymb}
\usepackage{amsthm}
\usepackage{newlfont}
\usepackage{stmaryrd}
\usepackage[]{titletoc}
\usepackage{lipsum}
\usepackage{soul}
\usepackage{txfonts}
\usepackage{mathtools}
\usepackage{mathrsfs}
\usepackage{extarrows}
\usepackage{geometry}
\newlength{\defbaselineskip}
\newcommand{\setlinespacing}[1]%
{\setlength{\baselineskip}{#1 \defbaselineskip}}

\theoremstyle{plain}
\newtheorem{thm}{Theorem}[section]
\newtheorem{defn}[thm]{Definition}

\newtheorem{lem}[thm]{Lemma}

\newtheorem{rem}[thm]{Remark}

\newcommand{\bn}{\mathbb{B}_n}
\newcommand{\cn}{\mathbb{C}^n}

\newcommand{\clb}{\overline{\mathbb{B}_n}}

\newcommand{\Xzw}{|X(z,w)|}

\newcommand{\D}{\mathbb{D}}
\newcommand{\C}{\mathbb{C}}
\newcommand{\dw}{\delta(w)}
\newcommand{\intd}{\mathrm{d}}

\newcommand{\bfe}{\mathbf{e}}
\newcommand{\Hol}{\mathrm{Hol}}
\newcommand{\bpartial}{\bar{\partial}}
\newcommand{\PP}{\mathcal{P}}

\newcommand\blfootnote[1]{%
	\begingroup
	\renewcommand\thefootnote{}\footnote{#1}%
	\addtocounter{footnote}{-1}%
	\endgroup
}

\allowdisplaybreaks[4]

\makeatletter\@addtoreset{equation}{section} \makeatother
\begin{document}
	\title{On the $p$-essential normality of principal submodules of the Bergman module on strongly pseudoconvex domains}
	\author{Ronald G. Douglas, Kunyu Guo, Yi Wang}
	\date{}
	\maketitle
	\begin{abstract}
		In this paper, we show that under a mild condition, a principal submodule of the Bergman module on a bounded strongly pseudoconvex domain with smooth boundary in $\cn$ is $p$-essentially normal for all $p>n$. This improves a previous result by the first author and K. Wang, in which it was shown that any polynomial-generated principal submodule of the Bergman module on the unit ball $\bn$ is $p$-essentially normal for all $p>n$. As a consequence, we show that the submodule of $L_a^2(\bn)$ consisting of functions vanishing on an analytic subset of pure codimension $1$ is $p$-essentially normal for all $p>n$.
	\end{abstract}
	
	\section{Introduction}
	Let $\C[z_1,\ldots,z_n]$ be the ring of analytic polynomials of $n$ variables. For a Hilbert space $\mathcal{H}$, a homomorphism
	$$
	\Phi:~\C[z_1,\ldots,z_n]\to\mathcal{B}(\mathcal{H})
	$$
	defines a $\C[z_1,\ldots,z_n]$-module structure on $\mathcal{H}$. In this case we say $\mathcal{H}$ is a Hilbert module (over $\C[z_1,\ldots,z_n]$). A closed subspace $\PP\subseteq\mathcal{H}$ invariant under the module actions is called a submodule. It naturally inherits a Hilbert module structure by restriction:
	\[
	\Phi':\C[z_1,\ldots,z_n]\to\mathcal{B}(\mathcal{H}),\quad p\mapsto\Phi(p)|_{\PP}.
	\]
	 The quotient space $\mathcal{Q}=\mathcal{H}\ominus\PP\cong\mathcal{H}/\PP$, inherits a module structure by compression:
	 \[
	 \Phi'':\C[z_1,\ldots,z_n]\to\mathcal{B}(\mathcal{Q}),\quad p\mapsto Q\Phi(p)|_{\mathcal{Q}}.
	 \]
	 Here $Q$ denotes the projection operator onto $\mathcal{Q}$.

	The Hilbert module $\mathcal{H}$ is said to be \emph{essentially normal} if
the cross commutator $[\Phi(z_i),\Phi(z_j)^*]$ is compact for all $i,j=1,\ldots,n.$ If moreover the corss commutators are in the Schatten-$p$ class $\mathcal{S}^p$ for some $p>0$, then we say that $\mathcal{H}$ is \emph{$p$-essentially normal}.
	
	A well known example of essentially normal Hilbert module is the Bergman module $L_a^2(\Omega)$ on a bounded strongly pseudoconvex domain $\Omega$ in $\cn$ with smooth boundary. In fact it is $p$-essentially normal for all $p>n$. The module action on $L_a^2(\Omega)$ is given by pointwise multiplications
	\[
	\Phi(z_i)=M_{z_i},\quad M_{z_i}f=z_if,\quad f\in L_a^2(\Omega), ~i=1,\ldots,n.
	\]
	For an ideal $I$ in $\C[z_1,\ldots,z_n]$, denote $[I]$ its closure in $L_a^2(\Omega)$. Then $[I]$ is a Hilbert submodule of $L_a^2(\Omega)$. The operators $R_i=M_{z_i}|_{[I]}, ~i=1,\ldots,n$ give the module action on $[I]$. 
	
	In a sequence of papers \cite{psummable}\cite{quotients}, Arveson conjectured that all submodules obtained by taking closure of a homogenous polynomial ideal in the Drury-Arveson module $H_n^2$ on the unit ball $\bn$ are $p$-essentially normal for all $p>n$.  Using the graded structure, one can show that the conjecture is equivalent to its analogous version on the Bergman space $L_a^2(\bn)$. It was then pointed out by the first author in \cite{index} that, more generally, an essentially normal Hilbert module may carry important geometric information. Suppose $\mathcal{H}$ is an irreducible essentially normal Hilbert module. Let $\mathcal{T}$ be the $C^*$-algebra generated by its module actions $\{\Phi(p): p\in\C[z_1,\ldots,z_n]\}$ and the identity operator $I$. Then the operators in $\mathcal{T}$ are essentially commutative, giving a short exact sequence
	\[
	0\to\mathcal{K}\to\mathcal{T}\to C(X)\to0.
	\] 
	Here $X$ is the joint essential Taylor spectrum of the tuple $\big(\Phi(z_1),\ldots,\Phi(z_n)\big)$. By the BDF-theory (cf. \cite{BDF}), the exact sequence gives an element $[\mathcal{H}]$ in the K-homology group $K_1(X)$. For the Bergman module $L_a^2(\bn)$ it can be shown that $X=\mathbb{S}_n=\partial\bn$, the unit sphere, and $[L_a^2(\bn)]$ is the fundamental class in $K_1(\mathbb{S}_n)$. For a homogenous ideal $I$, if the submodule $[I]$ is essentially normal, then $[I]$ defines the same $K$-homology element as $L_a^2(\bn)$. Things are more interesting on the quotient module $[I]^\perp$, in which case the $K$-homology element lies in $K_1\big(Z(I)\cap\mathbb{S}_n\big)$. Several generalizations of Arveson's conjecture were listed in \cite{index}. In particular, for a quotient module $[I]^\perp$ the range of $p$ where $[I]^\perp$ is $p$-essentially normal is conjectured to be $p>\dim_{\C}Z(I)$. Also, questions on the $p$-essential normality of submodules/quotient modules over strongly pseudoconvex domains was raised. These are generally refereed to as the Arveson-Douglas Conjecture. Various effort have been invested into the study of this conjecture, \cite{douglaswang}\cite{eschmeier}\cite{fangxia}\cite{Guo}\cite{GWK}\cite{GWKquasi}
	\cite{GZhao}\cite{KS2012}\cite{Shalit}, to list just a few. See \cite{survey} for a more detailed survey on this conjecture.
	
	In \cite{douglaswang}, the first author and K. Wang proved the surprising result that a principal submodule generated by any polynomial in the Bergman space on the unit ball $\bn$ is $p$-essentially normal for all $p>n$. Later Fang and Xia \cite{fangxia} extended this result to more general spaces, including the Hardy space on $\bn$. These results suggest that the conjecture might be true under a very general setting.
	
	The main goal of this paper is to prove the following result.
	
	\begin{thm}[Theorem \ref{thm: main}]\label{thm: h}
		Suppose $\Omega\subseteq\cn$ is a bounded strongly pseudoconvex domain with smooth boundary, $h$ is a holomorphic function defined in an open neighborhood of $\clb$. Then the principal submodule generated by $h$,
		\[
		[h]=\overline{\{hf:~f\in L_a^2(\Omega)\}},
		\] 
		is $p$-essentially normal for all $p>n$.
	\end{thm}
	
	In \cite{douglaswang}, the authors took an approach which involves some global estimate in which the controlling constant increases to infinity as the degree of the generating polynomial increases to infinity. In this paper, we take a different approach, which involves a sub-mean-type inequality (see Theorem \ref{thm: key inequality}) that can be proved using local estimates. Once the inequality is proved, the main result follows from a relatively standard argument.
	
	For a complex zero variety $V$, we define the submodule $\PP_V$ consisting of functions in $L_a^2(\Omega)$ that vanish on $V$. A geometric version of the Arveson-Douglas Conjecture was raised by Engli\v{s} and Eschmeier in \cite{ee}. For a nice variety $V$ it was conjectured that the quotient module $\mathcal{Q}_V=\PP_V^\perp$ is $p$-essential normality for $p>\dim_{\C}V$. 
	
	Suppose $V$ is of pure dimension $n-1$ in an open neighborhood of $\overline{\Omega}$, and assume that $\Omega$ satisfies an additional topological assumption that ensures solvability of the second Cousin problem, then (see Lemma \ref{lem: P_V for contractible}) $V$ has a defining function $h$ so that
	\[
	\mathcal{P}_V=\{hf: f\in\Hol(\Omega), ~hf\in L_a^2(\Omega)\}.
	\]
	As a byproduct of our main result, we show the following.
	\begin{thm}[Theorem \ref{thm: codim 1 on bn}]\label{thm: V}
		Suppose $V$ is a pure $(n-1)$-dimensional analytic subset of an open neighborhood of $\clb$, then $V$ has a minimal defining function $h$ on an open neighborhood of $\clb$. Moreover, we have
		\begin{equation*}
			[h]=\PP_V=\{hf: f\in\Hol(\bn), hf\in L_a^2(\bn)\}.
		\end{equation*}
		In particular, the submodule $\PP_V$ and the quotient module $\PP_V^\perp$ are $p$-essentially normal for all $p>n$.
	\end{thm}
	
	This paper is organized as follows. In section 2 we introduce tools and notations involving strongly pseudoconvex domains. In section 3 we prove the sub-mean-type inequality mentioned above. This will contain most of the technicality in this paper. In section 4 we prove our main result using the inequality. In section 5 we study the submodule $\PP_V$ and prove Theorem \ref{thm: V}.
	
	The third author would like to thank professor Kai Wang in Fudan University for valuable discussions. We also want to thank professor Jingbo Xia and Quanlei Fang for reading a previous version of the paper in full details and giving valuable suggestions. We would also like to thank the anonymous referees for reading the paper carefully and providing important comments and suggestions.
	
	\section{Preliminaries}
	
	%Defining function, strongly pseudoconvex smooth boundary, Kobayashi distance
	In this section we introduce some notions and tools involving strongly pseudoconvex domains. Our definitions and lemmas come from \cite{feff}\cite{krantz}\cite{bmo}\cite{weighted}\cite{Range}.
	
	\begin{defn}\label{defn: pseudoconvex}
	\begin{enumerate}
		\item For $\Omega$ a bounded domain in $\cn$ with smooth boundary, we call $\rho(z)$ a \emph{defining function} for $\Omega$ provided
		\begin{itemize}
			\item[(i)] $\Omega=\{z\in\cn:~\rho(z)<0\}$ and $\rho(z)\in C^{\infty}(\cn)$.
			\item[(ii)] $|\mathrm{grad}\rho(z)|\neq0$ for all $z\in\partial \Omega$.
		\end{itemize}
	\item For $\Omega$ a \emph{bounded strongly pseudoconvex domain with smooth boundary} we mean that there are a defining function $\rho$ of $\Omega$ and a constant $k>0$ such that
	\begin{equation}\label{eqn: levi form condition}
	\sum_{i,j=1}^n\frac{\partial^2\rho(\zeta)}{\partial z_i\partial\bar{z}_j}\xi_i\bar{\xi}_j\geq k|\xi|^2,\quad\forall\zeta\in\partial\Omega, ~\forall\xi\in\cn.
	\end{equation}
In the sequel, when we say that $\rho$ is a defining function for a bounded strongly pseudoconvex domain $\Omega$ with smooth boundary, we mean that $\rho$ is a defining function of $\Omega$ satisfying \eqref{eqn: levi form condition}.
		\item Suppose $\Omega$ is a bounded strongly pseudoconvex domain with smooth boundary and $\rho$ a defining function. For a point $\zeta\in\partial\Omega$, the vector
		\begin{equation}
		\bpartial\rho(\zeta)=\big(\frac{\partial\rho(\zeta)}{\partial\bar{z}_1},\frac{\partial\rho(\zeta)}{\partial\bar{z}_2},\ldots,\frac{\partial\rho(\zeta)}{\partial\bar{z}_n}\big)^T
		\end{equation}
		is non-zero and represents the \emph{complex normal direction} at $\zeta$.
		The \emph{complex tangent space} $T^{\C}_\zeta(\partial\Omega)$ (cf. \cite{bmo}) at $\zeta$ is defined to be the orthogonal complement of $\bpartial\rho(\zeta)$, i.e.,
		\begin{equation*}
			T_\zeta^{\C}(\partial \Omega)=\bigg\{\xi\in\cn:~\sum_{j=1}^n\frac{\partial \rho(\zeta)}{\partial z_j}\xi_j=0\bigg\}.
		\end{equation*}
		The definition does not depend on the choice of a defining function.
		 Since $\partial\Omega$ is smooth, for $z$ close enough to $\partial\Omega$ there exists a unique point $\pi(z)\in\partial\Omega$ with $d(z,\pi(z))=d(z,\partial\Omega)$, $d$ being the Euclidean distance. Moreover, the map $z\mapsto\pi(z)$ is smooth. (cf. \cite[Theorem 4.8 (5) and Lemma 4.11]{FeCurvatureMeasure}). For such $z$, define the \emph{complex normal (tangent) direction} at $z$ to be the corresponding complex normal (tangent) direction at $\pi(z)$.
	\end{enumerate}
	\end{defn}

\begin{rem}\label{rem: different definitions}
In classic literature, for example, \cite{krantz} and \cite{Range}, strongly pseudoconvex domains, also called strictly Levi pseudoconvex domains, are defined by the weaker condition that \eqref{eqn: levi form condition} holds only for $\xi\in T^{\C}_\zeta(\partial\Omega), \zeta\in\partial\Omega$. There is a related type of domain called strictly pseudoconvex domain, for which one requires \eqref{eqn: levi form condition} but drops the requirement (ii) in Definition \ref{defn: pseudoconvex}. However, in the case when $\Omega$ is bounded with smooth boundary, these definitions coincide and can be strengthened into the one we give in Definition \ref{defn: pseudoconvex} (cf. \cite[Page 60]{Range}). This allows us to cite references on both strongly pseudoconvex domains and strictly pseudoconvex domains.
\end{rem}

\noindent{\bf Notations:} For convenience, we fix the following notations in this paper. Suppose $\Omega$ is a bounded strongly pseudoconvex domain with smooth boundary and $\rho$ is a defining function for $\Omega$.
\begin{enumerate}
	\item For $\delta>0$, denote
	\begin{equation*}
	\Omega_\delta=\{z\in \Omega:~d(z,\partial \Omega)<\delta\},
	\end{equation*}
and
\begin{equation*}
R_\delta=\{(z,w)\in\overline{\Omega}\times\overline{\Omega}: |\rho(z)|+|\rho(w)|+|z-w|<\delta\}.
\end{equation*}
Denote
\begin{equation*}
	\Gamma=\{(z,z): z\in\partial\Omega\}.
\end{equation*}
\item Denote $\delta(z)$ the Euclidean distance of a point $z$ to $\partial\Omega$.
\begin{equation*}
\delta(z)=d(z,\partial\Omega).
\end{equation*}
	\item For $r>0$ and $z\in\C$, denote $\Delta(z, r)$ the open disk centered at $z$ with radius $r$.
	\begin{equation*}
	\Delta(z, r)=\{w\in\C: |z-w|<r\}.
	\end{equation*}
Denote $\D=\Delta(0,1)$ the open unit disk. In general, $\mathbb{B}_k(z,r)$ denotes the open ball in $\C^k$ centered at $z\in\C^k$ with radius $r$. 
	\item Suppose an orthonormal basis of $\cn$ is chosen. For $r_1, r_2, \ldots, r_n>0$, denote $P_{r_1,\ldots, r_n}$ the polydisk in $\cn$ with radius $r_i$ at the $i$-th direction.
	\begin{equation*}
	P_{r_1,\ldots,r_n}=\Delta(0,r_1)\times\ldots\times\Delta(0,r_n)=\{z\in\cn: |z_i|<r_i, i=1,\ldots,n\}.
	\end{equation*}
\item In Definition \ref{defn: pseudoconvex}, for $\delta>0$ small enough and $w\in\Omega_\delta$, the complex normal direction and complex tangential direction at $w$ are defined to be those of $\pi(w)$. For $a, b>0$, denote $P_w(a,b)$ the polydisk centered at $w$ with radius $a$ in the complex normal direction and radius $b$ in each of the complex tangential directions. The ambiguity caused by different choices of basis in the complex tangent space does not cause trouble in most cases. In the more explicit estimates (eg. the proof of Lemma \ref{lem: key inequality local}), we will specify the choice of basis before using this notation.
\item We use the notations $\approx$, $\lesssim$ and $\gtrsim$ to denote relations ``up to a constant (constants)'' between positive scalars . For example, $A\approx B$ means there exist $0<c<C$ such that $cB<A<CB$. $A\lesssim B$ means there exists a constant $C>0$ so that $A<CB$. The notations $C, c$ will be used to denote general constants and may vary in different places.
\item For a positive integer $k$, denote $v_k$ the Lebesgue measure on $\C^k$. For an open set $A$ in $\C^k$, $|A|$ denotes its volume $v_k(A)$.
\end{enumerate}

	\begin{lem}{\cite[Lemma 8]{bmo}}\label{lem: rho approx delta}
		Let $\Omega$ be a bounded strongly pseudoconvex domain with smooth boundary. Fix any defining function $\rho$, then for $z$ in a neighborhood of $\overline{\Omega}$,
		\begin{equation*}
		|\rho(z)|\approx\delta(z).
		\end{equation*}
	\end{lem}
	For this reason, in most of our discussions, using either $|\rho(z)|$ or $\delta(z)$ does not make a difference. We will choose whichever is more convenient.
	
	\begin{defn}
		Let $\Omega\subseteq\cn$ be a bounded strongly pseudoconvex domain with smooth boundary. The \emph{Bergman space} $L_a^2(\Omega)$ consists of all holomorphic functions on $\Omega$ which are square integrable with respect to the Lebesgue measure $v_n$.
		\[
		L_a^2(\Omega)=\{f\in \Hol(\Omega):~\int_{\Omega}|f(z)|^2dv_n(z)<\infty\}.
		\]
		For $l\geq0$, one defines the \emph{weighted Bergman space} $L_{a,l}^2(\Omega)$ in a similar way.
		\[
		L_{a,l}^2(\Omega)=\{f\in \Hol(\Omega):~\int_{\Omega}|f(z)|^2|\rho(z)|^ldv_n(z)<\infty\}.
		\]
	\end{defn}
	Standard arguments show that the Bergman and weighted Bergman spaces are reproducing kernel Hilbert spaces. We use $K(z,w)$ and $K_l(z,w)$ to denote their reproducing kernels, i.e.,
	\[
	f(z)=\int_{\Omega}f(w)K(z,w)dv_n(w),\quad\forall f\in L_a^2(\Omega), ~ z\in\Omega,
	\]
	\[
	f(z)=\int_{\Omega}f(w)K_l(z,w)|\rho(w)|^ldv_n(w),\quad\forall f\in L_{a,l}^2(\Omega), ~ z\in\Omega.
	\]
	
	\begin{defn}
	Suppose $\Omega\subseteq\cn$ is a bounded strongly pseudoconvex domain with smooth boundary, $z\in\Omega$ and $\xi\in\cn$, the \emph{infinitesimal Kobayashi metric} (cf. \cite{krantz}\cite{lang}\cite{bmo}) of $\Omega$ is defined by
	\[
	F_K(z,\xi)=\inf\{\alpha>0:~\exists f\in\D(\Omega) \mbox{ with }f(0)=z \mbox{ and }f'(0)=\xi/\alpha\},
	\]
	where $\D(\Omega)$ denotes the set of all holomorphic mappings from the open unit disc $\D$ to $\Omega$. For any $C^1$ curve $\gamma(t):[0,1]\to\Omega$, we define the \emph{Kobayashi length} of $\gamma(t)$ as
	\[
	L_K(\gamma)=\int_0^1F_K(\gamma(t),\gamma'(t))dt.
	\]
	If $z, w\in\Omega$, we write
	\[
	\beta(z,w)=\inf_\gamma\{L_K(\gamma)\},
	\]
	 where the infimum is taken over all $C^1$ curves with $\gamma(0)=z$ and $\gamma(1)=w$. Then $\beta(z,w)$ is a complete metric $\Omega$.

	For $w\in \Omega$ and $r>0$, denote $E(w,r)$ the \emph{Kobayashi ball} of radius $r$.
	\[
	E(w,r)=\{z\in \Omega:~\beta(z,w)<r\}.
	\]
	\end{defn}
	
	For fixed $r>0$ and any $w\in\Omega$ close to the boundary, a Kobayashi ball centered at $w$ is comparable to a polydisk centered at $w$. More explicitly, the following lemma holds.
	
	\begin{lem}{\cite[Lemma 6]{bmo}}\label{lem: Kball equiv to polydisk}
		Let $\Omega\subset\cn$ be a bounded strongly pseudoconvex domain with smooth boundary and $r>0$. If $z\in \Omega_{\delta}$ with $\delta>0$ small enough, then there are constants $a_i$ and $b_i$, $i=1,2$ only depending on $r$ and $\Omega$, such that
		\begin{equation*}
			P_w(a_1\dw,b_1\dw^{1/2})\subseteq E(w,r)\subseteq P_w(a_2\dw,b_2\dw^{1/2}).
		\end{equation*}
		In particular, $v_n(E(w,r))\approx\dw^{n+1}$.
	\end{lem}
	
	\begin{defn}\label{defn: X F}
	Suppose $\Omega$ is a bounded strongly pseudoconvex domain with smooth boundary. Fix some defining function $\rho(z)$ of $\Omega$. Let
	\begin{flalign*}
		X(z,w)=-\rho(z)-\sum_{j=1}^n\frac{\partial\rho(z)}{\partial z_j}(w_j-z_j)-1/2\sum_{j,k=1}^n\frac{\partial^2\rho(z)}{\partial z_j\partial z_k}(w_j-z_j)(w_k-z_k)
	\end{flalign*}
and
	\begin{equation*}
		F(z,w)=|\rho(z)|+|\rho(w)|+|\mathrm{Im} X(z,w)|+|z-w|^2.
	\end{equation*}
	\end{defn}
	
	The functions $X, F$ are useful in the estimates of reproducing kernels.
		
	\begin{lem}[\cite{bmo}]\label{lem: Bergman kernel}
		There exists a $\delta>0$ such that
		\begin{equation*}
			|K(z,w)|\approx\Xzw^{-(n+1)},\quad (z,w)\in R_\delta.
		\end{equation*}
		Moreover, $K(z,w)\in C^{\infty}(\overline{\Omega}\times\overline{\Omega}\backslash\Gamma)$
	\end{lem}

	\begin{lem}[\cite{weighted} Theorem 2.3]\label{lem: weighted Bergman kernel}
		Let $\Omega$ be a bounded strongly pseudoconvex domain with smooth boundary in $\cn$. 
		Then for $l\geq0$ there exists a kernel $G_l(z,w)$ such that:
		\begin{itemize}
			\item[~~(i)] $G_l(z,w)\in C^{\infty}(\overline{\Omega}\times\overline{\Omega}\backslash\Gamma)$, $G_l(z,w)$ is holomorphic in $z$.
			\item[~(ii)] $G_l$ reproduces the holomorphic functions in $L_{a,l}^2(\Omega)$; i.e., for $f\in L_{a,l}^2(\Omega)$,
			\[
			f(z)=\int_{\Omega}G_l(z,w)f(w)|\rho(w)|^ldv_n(w).
			\]
			\item[(iii)] For some $\delta>0$,
			\begin{equation*}
			|G_l(z,w)|\approx\Xzw^{-(n+1+l)},\quad (z,w)\in R_\delta.
			\end{equation*}
		\end{itemize}
	\end{lem}

\begin{lem}\label{lem: X F approx}
	Let $\Omega$ be a bounded strongly pseudoconvex domain with smooth boundary, then for some $\delta>0$, the following quantities are equivalent on $R_\delta$. 
	\begin{itemize}
		\item[(1)] $|X(z,w)|$;
		\item[(2)] $|X(w,z)|$;
		\item[(3)] $F(z,w)$;
		\item[(4)] $F(w,z)$;
		\item[(5)] $|\rho(z)|+|\rho(w)|+|z-w|^2+\big|\langle w-z,\bpartial\rho(z)\rangle\big|$;
		\item[(6)] $|\rho(z)|+|\rho(w)|+|z-w|^2+\big|\langle w-z,\bpartial\rho(w)\rangle\big|$;
		\item[(7)] $|\rho(z)|+|\rho(w)|+|z-w|^2+\big|\langle w-z,\bpartial\rho(\pi(z))\rangle\big|$;
		\item[(8)] $|\rho(z)|+|\rho(w)|+|z-w|^2+\big|\langle w-z,\bpartial\rho(\pi(w))\rangle\big|$.
		\end{itemize}
Moreover, the equivalences from (3)-(8) hold for all of $(z,w)\in\Omega\times\Omega$.
\end{lem}

\begin{proof}
The equivalence that for some $\delta>0$,
\[
|X(z,w)|\approx F(z,w)\approx |\rho(z)|+|\rho(w)|+|z-w|^2+\big|\langle w-z,\bpartial\rho(z)\rangle\big|,\quad (z,w)\in R_\delta
\]
was used in, for example, \cite{feff}\cite{bmo}. We give a proof for completeness. Take the Taylor expansion of the defining function $\rho(w)$ at $z$:
\begin{flalign*}
\rho(w)=&\rho(z)+2\mathrm{Re}\sum_{j=1}^n\frac{\partial\rho(z)}{\partial z_j}(w_j-z_j)\\
&+\mathrm{Re}\sum_{j,k=1}^n\frac{\partial^2\rho(z)}{\partial z_j\partial z_k}(w_j-z_j)(w_k-z_k)+\sum_{j,k=1}^n\frac{\partial^2\rho(z)}{\partial z_j\partial\bar{z}_k}(w_j-z_j)\overline{(w_k-z_k)}+O(|w-z|^3).
\end{flalign*}
By Definition \ref{defn: X F} and the above,
\begin{flalign*}
&2\mathrm{Re}X(z,w)\\
=&-2\rho(z)-2\mathrm{Re}\sum_{j=1}^n\frac{\partial\rho(z)}{\partial z_j}(w_j-z_j)-\mathrm{Re}\sum_{j,k=1}^n\frac{\partial^2\rho(z)}{\partial z_j\partial z_k}(w_j-z_j)(w_k-z_k)\\
=&-\rho(w)-\rho(z)+\sum_{j,k=1}^n\frac{\partial^2\rho(z)}{\partial z_j\partial\bar{z}_k}(w_j-z_j)\overline{(w_k-z_k)}+O(|w-z|^3).
\end{flalign*}
By assumption, for some $\delta>0$,
\[
\sum_{j,k=1}^n\frac{\partial^2\rho(z)}{\partial z_j\partial\bar{z}_k}(w_j-z_j)\overline{(w_k-z_k)}\approx|w-z|^2,\quad (z,w)\in R_\delta.
\]
For $z, w\in\overline{\Omega}$, $-\rho(w)-\rho(z)\geq 0$. Choose $\delta>0$ small enough, then the above implies that for $(z,w)\in R_\delta$,
\begin{equation*}
\mathrm{Re}X(z,w)\geq0,
\end{equation*}
and
\begin{equation}\label{eqn: temp ReX}
|\mathrm{Re}X(z,w)|\approx|\rho(w)|+|\rho(z)|+|w-z|^2.
\end{equation}
It follows that for $(z,w)\in R_\delta$,
\begin{flalign}\label{eqn: temp X 1}
|X(z,w)|\approx&|\mathrm{Re}X(z,w)|+|\mathrm{Im}X(z,w)|\nonumber\\
\approx&|\rho(w)|+|\rho(z)|+|w-z|^2+|\mathrm{Im}X(z,w)|\\
=&F(z,w).\nonumber
\end{flalign}
By Definition \ref{defn: X F},
\begin{equation}\label{eqn: temp ImX}
	\mathrm{Im}X(z,w)=-\mathrm{Im}\sum_{j=1}^n\frac{\partial\rho(z)}{\partial z_j}(w_j-z_j)+O(|w-z|^2).
\end{equation}
Therefore by \eqref{eqn: temp ReX} and \eqref{eqn: temp ImX}, 
\begin{flalign}\label{eqn: temp X lesssim}
|X(z,w)|\approx&|\mathrm{Re}X(z,w)|+|\mathrm{Im}X(z,w)|\\
\lesssim&|\rho(w)|+|\rho(z)|+|w-z|^2+\bigg|\sum_{j=1}^n\frac{\partial\rho(z)}{\partial z_j}(w_j-z_j)\bigg|,\nonumber
\end{flalign}
and
\begin{flalign}\label{eqn: temp X gtrsim}
|X(z,w)|\approx&|\mathrm{Re}X(z,w)|+|\mathrm{Im}X(z,w)|\\
\gtrsim&|\rho(w)|+|\rho(z)|+|w-z|^2+\bigg|\mathrm{Im}\sum_{j=1}^n\frac{\partial\rho(z)}{\partial z_j}(w_j-z_j)\bigg|.\nonumber
\end{flalign}
On the other hand, by Definition \ref{defn: X F} and \eqref{eqn: temp ReX},
\begin{equation}\label{eqn: temp Re normal}
\bigg|\mathrm{Re}\sum_{j=1}^n\frac{\partial\rho(z)}{\partial z_j}(w_j-z_j)\bigg|\lesssim|\mathrm{Re}X(z,w)|+|\rho(z)|+|w-z|^2\lesssim|\mathrm{Re}X(z,w)|.
\end{equation}
By \eqref{eqn: temp X lesssim}, \eqref{eqn: temp X gtrsim} and \eqref{eqn: temp Re normal}, we have
\begin{flalign}\label{eqn: temp X 2}
|X(z,w)|\approx&|\rho(w)|+|\rho(z)|+|w-z|^2+\bigg|\sum_{j=1}^n\frac{\partial\rho(z)}{\partial z_j}(w_j-z_j)\bigg|\nonumber\\
=&|\rho(w)|+|\rho(z)|+|w-z|^2+|\langle w-z, \bpartial\rho(z)\rangle|.
\end{flalign}
Combining \eqref{eqn: temp X 1} and \eqref{eqn: temp X 2} gives that,
\[
|X(z,w)|\approx F(z,w)\approx |\rho(w)|+|\rho(z)|+|w-z|^2+|\langle w-z, \bpartial\rho(z)\rangle|,\quad (z,w)\in R_\delta.
\]
In other words, (1)$\approx$(3)$\approx$(5).
By the smoothness assumption, the quantities (5)-(8) differ at most by constant multiples of 
\[
|w-z|^2+|\rho(z)|+|\rho(w)|.
\]
Thus they are equivalent, i.e., (1)$\approx$(3)$\approx$(5)-(8). Finally, switching $z$ and $w$ turns (5) into (6) which are equivalent. So (1)$\approx$(2), (3)$\approx$(4). This proves the equivalences on $R_\delta$ for $\delta$ small enough. By Definition \ref{defn: X F}, for $(z,w)$ outside $R_\delta$, the quantities $F(z,w)$ and $ F(w,z)$ are bounded both from above and away from zero. So do the quantities (5)-(8). This implies that the equivalence (3)-(8) holds for all $z, w\in\Omega$. This completes the proof.
\end{proof}

\begin{rem}\label{rem: Kzw Glzw lesssim Fzw}
	Take $\delta>0$ such that Lemma \ref{lem: Bergman kernel}, Lemma \ref{lem: weighted Bergman kernel} and Lemma \ref{lem: X F approx} hold for $R_\delta$. Then
	\[
	|K(z,w)|\approx F(z,w)^{-(n+1)},\quad|G_l(z,w)|\approx F(z,w)^{-(n+1+l)},\quad\forall (z,w)\in R_\delta.
	\]
	Since $K(z,w), G_l(z,w)\in C^\infty(\overline{\Omega}\times\overline{\Omega}\backslash\Gamma)$, the functions $|G_l(z,w)|, |K(z,w)|$ are bounded on $\overline{\Omega}\times\overline{\Omega}\backslash R_\delta$. Also, the function $F(z,w)$ is bounded. Thus
	\begin{equation}\label{eqn: Kzw lesssim Fzw}
	|K(z,w)|\lesssim F(z,w)^{-(n+1)},\quad\forall z, w\in\Omega,
	\end{equation}
and
\begin{equation}\label{eqn: Glzw lesssim Fzw}
|G_l(z,w)|\lesssim F(z,w)^{-(n+1+l)},\quad\forall z, w\in\Omega.
\end{equation}
\end{rem}

The following lemma comes from the proof of \cite[Theorem 12]{bmo}.

\begin{lem}\label{lem: rzrw beta leq r}
	Fix some $r>0$. Then
	\begin{equation*}
		|\rho(z)|\approx|\rho(w)|,\quad\forall z, w\in \Omega,~ \beta(z,w)<r.
	\end{equation*}
\end{lem}

\begin{lem}\label{lem: Fzw beta leq r}
	Fix some $r>0$. Then 
	\begin{equation*}
	F(z,\lambda)\approx F(w,\lambda),\quad\forall z, w, \lambda\in\Omega, ~ \beta(z,w)<r.
	\end{equation*}
\end{lem}

\begin{proof}
By Lemma \ref{lem: X F approx}, for $z, w, \lambda\in\Omega$,
\begin{equation*}
F(z,\lambda)\approx|\rho(z)|+|\rho(\lambda)|+|z-\lambda|^2+\big|\langle z-\lambda,\bpartial\rho(\pi(z))\rangle\big|,
\end{equation*}
\begin{equation*}
	F(w,\lambda)\approx|\rho(w)|+|\rho(\lambda)|+|w-\lambda|^2+\big|\langle w-\lambda,\bpartial\rho(\pi(w))\rangle\big|.
\end{equation*}
By Lemma \ref{lem: Kball equiv to polydisk}, if $\beta(z,w)<r$,
\begin{equation*}
\big|\langle w-z,\bpartial\rho(\pi(z))\rangle\big|\lesssim|\rho(z)|\approx|\rho(w)|,\quad|w-z|^2\lesssim|\rho(z)|\approx|\rho(w)|.
\end{equation*}
Therefore by Lemma \ref{lem: rzrw beta leq r} and the above, if $\beta(z,w)<r$, then
\begin{flalign*}
F(z,\lambda)\approx&|\rho(z)|+|\rho(\lambda)|+|z-\lambda|^2+\big|\langle z-\lambda,\bpartial\rho(\pi(z))\big|\\
\lesssim&|\rho(w)|+|\rho(\lambda)|+\big(|w-\lambda|^2+|w-z|^2\big)+\big|\langle w-\lambda,\bpartial\rho(\pi(z))\big|+\big|\langle w-z,\bpartial\rho(\pi(z))\rangle\big|\\
\lesssim&|\rho(w)|+|\rho(\lambda)|+|w-\lambda|^2+\big|\langle w-\lambda,\bpartial\rho(\pi(z))\big|\\
\leq&|\rho(w)|+|\rho(\lambda)|+|w-\lambda|^2+\big|\langle w-\lambda,\bpartial\rho(\pi(w))\big|+\big|\langle w-\lambda,\bpartial\rho(\pi(w))-\bpartial\rho(\pi(z))\big|\\
\lesssim&|\rho(w)|+|\rho(\lambda)|+|w-\lambda|^2+\big|\langle w-\lambda,\bpartial\rho(\pi(w))\big|+|w-\lambda|^2+|w-z|^2\\
\lesssim&|\rho(w)|+|\rho(\lambda)|+|w-\lambda|^2+\big|\langle w-\lambda,\bpartial\rho(\pi(w))\big|\\
\approx&F(w,\lambda).
\end{flalign*}
This proves $F(z,\lambda)\lesssim F(w,\lambda).$ Switching $z$ and $w$ gives the reverse inequality. This completes the proof.
\end{proof}

\begin{lem}\label{lem: biholomorphic}
Suppose $\Omega\subset\cn$ is a bounded strongly pseudoconvex domain with smooth boundary and $\Phi$ is a biholomorphic map on a neighborhood of $\overline{\Omega}$. Let $\Omega'=\Phi(\Omega)$. Denote $F$ and $F'$ the functions defined as in Definition \ref{defn: X F} for $\Omega$ and $\Omega'$ respectively. Then
\begin{equation*}
F(z,w)\approx F'(\Phi(z),\Phi(w)),\quad z, w\in\Omega.
\end{equation*}
\end{lem}

\begin{proof}
	Fix a defining function $\rho$. Then $\rho\circ\Phi^{-1}$ is a defining function for $\Omega'$. By Lemma \ref{lem: rho approx delta} and Lemma \ref{lem: X F approx},
	\begin{equation}\label{eqn: temp F}
	F(z,w)\approx\delta(z)+\delta(w)+|z-w|^2+\big|\langle w-z,\bpartial\rho(z)\rangle\big|,
	\end{equation}
and
	\begin{equation}\label{eqn: temp F'}
		\big|F'(\Phi(z),\Phi(w))\big|\\
		\approx\delta\big(\Phi(z)\big)+\delta\big(\Phi(w)\big)+\big|\Phi(z)-\Phi(w)\big|^2+\big|\langle \Phi(w)-\Phi(z),\bpartial\big(\rho\circ\Phi^{-1}\big)(\Phi(z))\rangle\big|.
	\end{equation}
	Since $\Phi$ is biholomorphic, $\Phi$ is bi-Lipschitz. Therefore
	\[
	\delta(z)\approx\delta\big(\Phi(z)\big),\quad\delta(w)\approx\delta\big(\Phi(w)\big),\quad |z-w|\approx\big|\Phi(z)-\Phi(w)\big|.
	\]
	We look at the last terms of \eqref{eqn: temp F} and \eqref{eqn: temp F'} to deduce
	\begin{flalign*}
		&\big\langle \Phi(w)-\Phi(z),\bpartial\big(\rho\circ\Phi^{-1}\big)(\Phi(z))\big\rangle\\
		=&\sum_{j=1}^n\frac{\partial\rho\circ\Phi^{-1}(\Phi(z))}{\partial z_j}\big(\Phi_j(w)-\Phi_j(z)\big)\\
		=&\sum_{j=1}^n\sum_{i=1}^n\frac{\partial\rho(z)}{\partial z_i}\frac{\partial\Phi^{-1}_i(\Phi(z))}{\partial z_j}\bigg(\sum_{k=1}^n\frac{\partial\Phi_j(z)}{\partial z_k}(w_k-z_k)+O(|w-z|^2)\bigg)\\
		=&\sum_{i=1}^n\frac{\partial\rho(z)}{\partial z_i}\sum_{k=1}^n\bigg(\sum_{j=1}^n\frac{\partial\Phi_i^{-1}(\Phi(z))}{\partial z_j}\frac{\partial \Phi_j(z)}{\partial z_k}\bigg)(w_k-z_k)+O(|w-z|^2)\\
		=&\sum_{i=1}^n\frac{\partial\rho(z)}{\partial z_i}\sum_{k=1}^n\delta_{ik}(w_k-z_k)+O(|w-z|^2)\\
		=&\sum_{i=1}^n\frac{\partial\rho(z)}{\partial z_i}(w_i-z_i)+O(|w-z|^2)\\
		=&\langle w-z, \bpartial\rho(z)\rangle+O(|w-z|^2).
	\end{flalign*}
	From this it is clear that $F(z,w)\approx F'(\Phi(z),\Phi(w))$. This completes the proof.
\end{proof}

We will also use the following Rudin-Forelli type estimates on $\Omega$.
	
	\begin{lem}{\cite[Lemma 2.7]{weighted}}\label{lem: Rudin Forelli}
		Let $\Omega$ be a bounded strongly pseudoconvex domain in $\cn$ with smooth boundary. Let $a\in\mathbb{R}$, $\nu>-1$, then
		\[
		\int_{\Omega}\frac{|\rho(w)|^{\nu}}{F(z,w)^{n+1+\nu+a}}dv_n(w)\approx
		\begin{cases}
			1&\text{if }a<0\\
			\log|\rho(z)|^{-1}&\text{if }a=0\\
			|\rho(z)|^{-a}&\text{if }a>0
		\end{cases}.
		\]
	\end{lem}

	\section{An Inequality}\label{sec: an inequality}
	The goal of this section is to prove the following estimate, which plays a key role in the proof of our main theorem.
	\begin{thm}\label{thm: key inequality}
		Suppose $\Omega\subseteq\cn$ is a bounded strongly pseudoconvex domain with smooth boundary and $h$ is a holomorphic function defined in an open neighborhood of $\overline{\Omega}$. Then there exist constants $N>0, C>0$ such that for any $w,z\in \Omega$ and $f\in \Hol(\Omega)$,
		\begin{equation}\label{eqn: key inequality}
			|h(z)f(w)|\leq C\frac{F(z,w)^N}{|\rho(w)|^{N+n+1}}\int_{E(w,1)}|h(\lambda)f(\lambda)|\intd v_n(\lambda).
		\end{equation}
	Here the function $F$ is defined as in Definition \ref{defn: X F}.
	\end{thm}
	Let us first give an outline of the proof. By Lemma \ref{lem: Kball equiv to polydisk}, for $w$ close to the boundary, $E(w,1)$ is comparable with some $P_w(a\delta(w),b\delta(w)^{1/2})$. A large part of this section is to show that the following inequality holds for $(z,w)\in R_\delta$, for some $\delta>0$.
	\begin{flalign}\label{eqn: ln|h| inequality}
		\log|h(z)|\leq& \frac{1}{|P_w(a\delta(w),b\delta(w)^{1/2})|}\int_{P_w(a\delta(w),b\delta(w)^{1/2})}\log|h(\lambda)|\intd v_n(\lambda)\\
		&~~~+N\log\frac{F(z,w)}{|\rho(w)|}+\text{some constant}.\nonumber
	\end{flalign}
	For any $f\in\Hol(\Omega)$, the function $\log|f|$ is pluri-subharmonic. So 
	\begin{equation}\label{eqn: ln|f| inequality}
		\log|f(w)|\leq \frac{1}{|P_w(a\delta(w),b\delta(w)^{1/2})|}\int_{P_w(a\delta(w),b\delta(w)^{1/2})}\log|f(\lambda)|\intd v_n(\lambda)
	\end{equation}
Adding up \eqref{eqn: ln|h| inequality} and \eqref{eqn: ln|f| inequality}, and then applying Jensen's inequality yields
		\begin{equation*}
		|h(z)f(w)|\lesssim\bigg(\frac{F(z,w)}{|\rho(w)|}\bigg)^{N}\cdot\frac{1}{|P_w(a\delta(w),b\delta(w)^{1/2})|}\int_{P_w(a\delta(w),b\delta(w)^{1/2})}|h(\lambda)f(\lambda)|\intd v_n(\lambda).
	\end{equation*}
From this we get the following local version of Theorem \ref{thm: key inequality}.
\begin{lem}\label{lem: key inequality Kdelta}
		Suppose $\Omega\subseteq\cn$ is a bounded strongly pseudoconvex domain with smooth boundary and $h$ is a holomorphic function defined in an open neighborhood of $\overline{\Omega}$. Then there exist $\delta>0$, $N>0, C>0$ such that \eqref{eqn: key inequality} holds for any $(z, w)\in R_\delta\cap\Omega\times\Omega$ and $f\in\Hol(\Omega)$.
	\end{lem}
Once Lemma \ref{lem: key inequality Kdelta} is established, Theorem \ref{thm: key inequality} follows from a relatively easy argument.

We start our proof. The following estimate holds by direct computation.
	\begin{lem}\label{lem: key inequality at dim 1 degree 1}
		For any $z, a\in\C$ and $r>0$,
		\begin{equation*}
			\log|z-a|<\frac{1}{|\Delta(0,r)|}\int_{\Delta(0,r)}\log|\lambda-a|\intd v_1(\lambda)+\log\frac{r+|z|}{r}+1.
		\end{equation*}
	\end{lem}
	
	\begin{proof}
		We split the proof into two cases.
		
		\noindent\textbf{Case 1: $|a|\geq r$.}
		
		In this case the function $\log|\lambda-a|$ is harmonic in $\Delta(0,r)$. Therefore
		\[
		\frac{1}{|\Delta(0,r)|}\int_{\Delta(0,r)}\log|\lambda-a|\intd v_1(\lambda)=\log|a|.
		\]
		It follows that
		\[
		\log|z-a|-\frac{1}{|\Delta(0,r)|}\int_{\Delta(0,r)}\log|\lambda-a|\intd v_1(\lambda)-\log\frac{r+|z|}{r}=\log\frac{|z-a|r}{|a|(r+|z|)}.
		\]
		Set $M=\max\{r+|z|, |a|\}$ and $m=\min\{r+|z|, |a|\}$. Then
		\[
		|z-a|\leq|z|+|a|\leq 2\max\{|z|, |a|\}\leq 2\max\{r+|z|, |a|\}=2M,
		\]
		and
		\[
		r\leq\min\{r+|z|, |a|\}=m.
		\]
		Therefore
		\[
		\frac{|z-a|r}{|a|(r+|z|)}=\frac{|z-a|r}{M m}\leq 2.
		\]
		This shows that
		\[
		\log|z-a|-\frac{1}{|\Delta(0,r)|}\int_{\Delta(0,r)}\log|\lambda-a|\intd v_1(\lambda)-\log\frac{r+|z|}{r}=\log\frac{|z-a|r}{|a|(r+|z|)}\leq\log 2<1.
		\]
		This proves the inequality in the case 1.
		
		~
		
		\noindent\textbf{Case 2: $|a|<r$.}
		
		\begin{flalign*}
			&\int_{\Delta(0,r)}\log|\lambda-a|\intd v_1(\lambda)\\
			=&\int_{\Delta(a,r)}\log|\lambda-a|\intd v_1(\lambda)-\int_{\Delta(a,r)\backslash\Delta(0,r)}\log|\lambda-a|\intd v_1(\lambda)+\int_{\Delta(0,r)\backslash\Delta(a,r)}\log|\lambda-a|\intd v_1(\lambda)\\
			\geq&\int_{\Delta(a,r)}\log|\lambda-a|\intd v_1(\lambda)-\int_{\Delta(a,r)\backslash\Delta(0,r)}\log r\intd v_1(\lambda)+\int_{\Delta(0,r)\backslash\Delta(a,r)}\log r\intd v_1(\lambda)\\
			=&\int_{\Delta(a,r)}\log|\lambda-a|\intd v_1(\lambda)\\
			=&\int_0^r2\pi s\log s\intd s\\
			=&\pi r^2\big(\log r-\frac{1}{2}\big).
		\end{flalign*}
		It follows that
		\begin{flalign*}
			&\log|z-a|-\frac{1}{|\Delta(0,r)|}\int_{\Delta(0,r)}\log|\lambda-a|\intd v_1(\lambda)-\log\frac{r+|z|}{r}\\
			\leq&\log|z-a|-\big(\log r-\frac{1}{2}\big)-\log\frac{r+|z|}{r}\\
			=&\log\frac{|z-a|}{r+|z|}+\frac{1}{2}\\
			\leq&\log\frac{|a|+|z|}{r+|z|}+\frac{1}{2}\\
			\leq&\log 1+\frac{1}{2}\\
			<&1.
		\end{flalign*}
		This proves the inequality in case 2.
		The proof is complete.
	\end{proof}

	\begin{lem}\label{lem: key inequality in dim 1 degree d}
		Suppose $p$ is a one-variable polynomial of degree less than or equal to $d$. Then for any $r>0$ and $z\in\C$,
		\begin{equation*}
			\log|p(z)|\leq\frac{1}{|\Delta(0,r)|}\int_{\Delta(0,r)}\log|p(\lambda)|\intd v_1(\lambda)+d\log\frac{r+|z|}{r}+d.
		\end{equation*}
	\end{lem}
	
	\begin{proof}
The case $p\equiv0$ is trivial. Suppose $p\nequiv0$. Then by assumption,
		\[
		p(\lambda)=c\prod_{i=1}^k(\lambda-a_i),\quad k\leq d,~ c\neq0.
		\]
		Therefore by Lemma \ref{lem: key inequality at dim 1 degree 1},
		\begin{flalign*}
			\log|p(z)|=&\log|c|+\sum_{i=1}^k\log|z-a_i|\\
			<&\log|c|+\sum_{i=1}^k\bigg[\frac{1}{|\Delta(0,r)|}\int_{\Delta(0,r)}\log|\lambda-a_i|\intd v_1(\lambda)+\log\frac{r+|z|}{r}+1\bigg]\\
			=&\frac{1}{|\Delta(0,r)|}\int_{\Delta(0,r)}\log|p(\lambda)|\intd v_1(\lambda)+k\log\frac{r+|z|}{r}+k\\
			\leq&\frac{1}{|\Delta(0,r)|}\int_{\Delta(0,r)}\log|p(\lambda)|\intd v_1(\lambda)+d\log\frac{r+|z|}{r}+d.
		\end{flalign*}
		This completes the proof.
	\end{proof}

For a $n$-variable polynomial $p\in\C[z_1,z_2,\ldots,z_n]$, the degree of $p$ is the highest degree of the monomials contained in $p$, i.e., for $p(z)=\sum a_\alpha z^\alpha$,
\[
\mathrm{degree} p=\{|\alpha|:~a_\alpha\neq0\}.
\]

\begin{lem}\label{lem: key inequality dim n degree d}
Suppose $p$ is a $n$-variable polynomial, i.e., $p\in\C[z_1,z_2,\ldots,z_n]$. Assume that $p$ is of degree less than or equal to $d$. Then for any $r_1, \ldots, r_n>0$ and $z\in\cn$,
\begin{equation*}
\log|p(z)|\leq\frac{1}{|P_{r_1,\ldots,r_n}|}\int_{P_{r_1,\ldots,r_n}}\log|p(\lambda)|\intd v_n(\lambda)+d\sum_{i=1}^n\log\frac{r_i+|z_i|}{r_i}+nd.
\end{equation*}
\end{lem}

\begin{proof}
The inequality follows from Lemma \ref{lem: key inequality in dim 1 degree d} by induction.
\begin{flalign*}
&\log|p(z)|\\
\leq&\frac{1}{|\Delta(0,r_1)|}\int_{\Delta(0,r_1)}\log|p(\lambda_1,z_2,\ldots,z_n)|\intd v_1(\lambda_1)+d\log\frac{r_1+|z_1|}{r_1}+d\\
\leq&\frac{1}{|\Delta(0,r_1)|}\int_{\Delta(0,r_1)}\bigg\{\frac{1}{|\Delta(0,r_2)|}\int_{\Delta(0,r_2)}\log|p(\lambda_1,\lambda_2,z_3,\ldots,z_n)|\intd v_1(\lambda_2)+d\log\frac{r_2+|z_2|}{r_2}+d\bigg\}\intd v_1(\lambda_1)\\
&+d\log\frac{r_1+|z_1|}{r_1}+d\\
=&\frac{1}{\pi^2r_1^2r_2^2}\int_{\Delta(0,r_1)\times\Delta(0,r_2)}\log|p(\lambda_1,\lambda_2,z_3,\ldots,z_n)|\intd v_2(\lambda_1,\lambda_2)+d\frac{r_2+|z_2|}{r_2}+d\log\frac{r_1+|z_1|}{r_1}+2d\\
&\ldots\\
\leq&\frac{1}{\pi^nr_1^2\ldots r_n^2}\int_{P_{r_1,\ldots,r_n}}\log|p(\lambda)|\intd v_n(\lambda)+d\sum_{i=1}^n\log\frac{r_i+|z_i|}{r_i}+nd\\
=&\frac{1}{|P_{r_1,\ldots,r_n}|}\int_{P_{r_1,\ldots,r_n}}\log|p(\lambda)|\intd v_n(\lambda)+d\sum_{i=1}^n\log\frac{r_i+|z_i|}{r_i}+nd.
\end{flalign*}
This completes the proof.
\end{proof}

	\begin{lem}\label{lem: W decomp}
		Suppose $R>0$ and $H(\xi,w)$ is a continuous function on $\big(\Delta(0,R)\big)^{n}\times\big(\Delta(0,R)\big)^n$. Assume that
		\begin{enumerate}
			\item for fixed $w$, $H(\xi, w)$ is holomorphic in $\xi$;
			\item for each $i=1, \ldots, n$, the one-variable holomorphic function $H(\xi_i e_i, 0)$ is not identically zero. Denote $d_i$ the multiplicity of $H(\xi_i e_i, 0)$ at $\xi_i=0$.
		\end{enumerate}
		Then there exist $0<r<R$, $0<c<C<\infty$ such that for any $i=1, 2, \ldots, n$,
		\begin{equation*}
			H(\xi, w)=P_i(\xi, w)\varphi_i(\xi, w),\quad\forall \xi, w\in\big(\Delta(0,r)\big)^n.
		\end{equation*}
		The functions $P_i, \varphi_i$ satisfy the following.
		\begin{enumerate}
			\item For each fixed $\xi_1, \ldots, \xi_{i-1}, \xi_{i+1}, \ldots, \xi_n, w_1, \ldots, w_n \in\Delta(0,r)$, $P_i(\xi, w)$ is a polynomial in $\xi_i$ of degree $d_i$.
			\item For any $\xi, w\in\big(\Delta(0,r)\big)^n$,
			\begin{equation*}
				c\leq|\varphi_i(\xi, w)|\leq C.
			\end{equation*}
		\end{enumerate}
	\end{lem}
	
	\begin{proof}
		We consider $i=1$. By assumption, $H(\xi_1e_1,0)$ is a one-variable holomorphic function of multiplicity $d_1$ at $0$. Choose $0<s_1<R$ so that
		\begin{equation*}
			H(\xi_1e_1, 0)\neq0,\quad \frac{s_1}{2}\leq|\xi_1|\leq s_1.
		\end{equation*}
		By continuity, there exists $0<s_2<R$ so that
		\begin{equation}\label{eqn: temp 2}
			H(\xi, w)\neq0,\quad \frac{s_1}{2}\leq|\xi_1|\leq s_1, ~ \xi'\in\overline{\big(\Delta(0,s_2)\big)^{n-1}}, ~ w\in\overline{\big(\Delta(0,s_2)\big)^n}.
		\end{equation}
		Here we denote $\xi'=(\xi_2,\ldots,\xi_n)$.
		Viewing $H(\xi, w)$ as a parameterized family of one-variable holomorphic functions in $\xi_1$, by \eqref{eqn: temp 2}, the continuity of $H$, and Rouch\'{e}'s Theorem, we know that for each $\xi'\in\overline{\big(\Delta(0,s_2)\big)^{n-1}}, ~ w\in\overline{\big(\Delta(0,s_2)\big)^n}$, $H(\xi,w)$ as a one-variable holomorphic function of $\xi_1$, has exactly $d_1$ zeros inside $\Delta(0,s_1)$, counting multiplicity.  Denote the $d_1$ zeros to be 
		\[
		a_1(\xi',w), a_2(\xi',w), \ldots, a_{d_1}(\xi', w).
		\]
		Then by \eqref{eqn: temp 2},
		\begin{equation}\label{eqn: temp 3}
			|a_j(\xi', w)|<\frac{s_1}{2}, \quad j=1,\ldots, d_1, ~ \xi'\in\overline{\big(\Delta(0,s_2)\big)^{n-1}}, ~ w\in\overline{\big(\Delta(0,s_2)\big)^n}.
		\end{equation}
		Define
		\begin{equation*}
			P_1(\xi, w)=\prod_{j=1}^{d_1}(\xi_1-a_j(\xi',w)),\quad \varphi_1(\xi, w)=\frac{H(\xi, w)}{P_1(\xi, w)}
		\end{equation*}
		for $\xi_1\in\Delta(0,s_1), \xi'\in\big(\Delta(0,s_2)\big)^{n-1}, w\in\big(\Delta(0,s_2)\big)^n.$
		Since for fixed $\xi'$ and $w$, the functions $P_1$ and $H$ are holomorphic functions in $\xi_1$ of exactly the same zeros in $\Delta(0,s_1)$, the function $\varphi_1$ is well-defined even at points of $P_1=0$ and is also holomorphic in $\xi_1$. By \eqref{eqn: temp 2}, there exist $0<m<M$ such that
		\begin{equation}\label{eqn: temp 5}
			m\leq|H(\xi,w)|\leq M ~\text{ for }~ |\xi_1|=s_1, ~ \xi'\in\big(\Delta(0,s_2)\big)^{n-1}, ~ w\in\big(\Delta(0,s_2)\big)^n.
		\end{equation}
		Also, by \eqref{eqn: temp 3}, 
		\begin{equation}\label{eqn: temp 6}
			(s_1/2)^{d_1}\leq|P_1(\xi,w)|\leq(2s_1)^{d_1} ~ \text{ for } ~ |\xi_1|=s_1, ~ \xi'\in\big(\Delta(0,s_2)\big)^{n-1}, ~ w\in\big(\Delta(0,s_2)\big)^n.
		\end{equation}
		Thus by \eqref{eqn: temp 5} and \eqref{eqn: temp 6},
		\begin{equation}\label{eqn: temp 7}
			|\varphi_1(\xi,w)|\leq\frac{M}{(s_1/2)^{d_1}},\quad\frac{1}{|\varphi_1(\xi,w)|}\leq\frac{(2s_1)^{d_1}}{m}
		\end{equation}
		for $|\xi_1|=s_1, ~ \xi'\in\big(\Delta(0,s_2)\big)^{n-1}, ~ w\in\big(\Delta(0,s_2)\big)^n.$ By the Maximum Principle, the inequalities in \eqref{eqn: temp 7} also hold for $|\xi_1|<s_1, \xi'\in\big(\Delta(0,s_2)\big)^{n-1}, w\in\big(\Delta(0,s_2)\big)^n.$ Take $r_1=\min\{s_1, s_2\}$, $c_1=\frac{M}{(s_1/2)^{d_1}}$ and $C_1=\frac{m}{(2s_1)^{d_1}}$. Then the decomposition $H=P_1\varphi_1$ is defined on $\big(\Delta(0,r_1)\big)^{n}\times\big(\Delta(0,r_1)\big)^{n}$ and
		\[
		c_1\leq|\varphi_1(\xi,w)|\leq C_1, \quad (\xi, w)\in\big(\Delta(0,r_1)\big)^n\times\big(\Delta(0,r_1)\big)^n.
		\]
		For $i=2,\ldots,n$ we can similarly construct decomposition $H=P_i\varphi_i$ on some $\big(\Delta(0,r_i)\big)^{n}\times\big(\Delta(0,r_i)\big)^{n}$, such that $P_i$ is a polynomial of degree $d_i$ in $\xi_i$, and $c_i\leq|\varphi_i|\leq C_i$ on $\big(\Delta(0,r_i)\big)^{n}\times\big(\Delta(0,r_i)\big)^{n}$. Finally, take $r=\min\{r_1,\ldots, r_n\}$, $c=\min\{c_1,\ldots, c_n\}$ and $C=\max\{C_1,\ldots,C_n\}$. This completes the proof.
	\end{proof}
	
	\begin{rem}
		Lemma \ref{lem: W decomp} is essentially a modification of the Weierstrass Preparation Theorem.
		In fact, the proof of the Weierstrass Preparation Theorem shows that the functions $P_i$ and $\varphi_i$ are holomorphic in $\xi$. However, that fact is not needed in this paper.
	\end{rem}
	
	\begin{lem}\label{lem: key inequality ln}
		Let $H(\xi,w)$ and $r, C, c$ be as in Lemma \ref{lem: W decomp}. Then for any $0<r_1,\ldots, r_n<r$ and any $\xi, w\in\big(\Delta(0,r)\big)^n$,
		\begin{equation*}
			\log|H(\xi,w)|\leq\frac{1}{|P_{r_1,\ldots,r_n}|}\int_{P_{r_1,\ldots,r_n}}\log|H(\eta,w)|\intd v_n(\eta)+\sum_{i=1}^nd_i\log\frac{r_i+|\xi_i|}{r_i}+\sum_{i=1}^n d_i+n\log\frac{C}{c}.
		\end{equation*}
	\end{lem}	
	
	\begin{proof}
		By Lemma \ref{lem: key inequality in dim 1 degree d} and Lemma \ref{lem: W decomp}, for $\xi, w\in\big(\Delta(0,r)\big)^n$ and $i=1,\ldots,n$,
		\begin{flalign*}
			&\log|H(\xi,w)|\\
			=&\log|P_i(\xi,w)|+\log|\varphi_i(\xi,w)|\\
			\leq&\bigg\{\frac{1}{|\Delta(0,r_i)|}\int_{\Delta(0,r_i)}\log|P_i(\xi_1,\ldots,\xi_{i-1},\eta_i,\xi_{i+1},\ldots,\xi_n, w)|\intd v_1(\eta_i)+d_i\log\frac{r_i+|\xi_i|}{r_i}+d_i\bigg\}+\log C\\
			\leq&\frac{1}{|\Delta(0,r_i)|}\int_{\Delta(0,r_i)}\log|H(\xi_1,\ldots,\xi_{i-1},\eta_i,\xi_{i+1},\ldots,\xi_n, w)|\intd v_1(\eta_i)+d_i\log\frac{r_i+|\xi_i|}{r_i}+d_i+\log\frac{C}{c}.
		\end{flalign*}
		Since $r_i<r$,  each $(\xi_1,\ldots,\xi_{i-1},\eta_i,\xi_{i+1},\ldots,\xi_n)$ in the integrand above still belong to $\big(\Delta(0,r)\big)^n$. This allows us to iterate the inequality over $i=1, 2,\ldots, n$. This gives
		\begin{flalign*}
			&\log|H(\xi,w)|\\
			\leq&\frac{1}{|\Delta(0,r_1)|}\int_{\Delta(0,r_1)}\log|H(\eta_1,\xi_2,\ldots,\xi_n,w)|\intd v_1(\eta_1)+d_1\log\frac{r_1+|\xi_1|}{r_1}+d_1+\log\frac{C}{c}\\
			\leq&\frac{1}{|\Delta(0,r_1)|}\int_{\Delta(0,r_1)}\bigg\{\frac{1}{|\Delta(0,r_2)|}\int_{\Delta(0,r_2)}\log|H(\eta_1,\eta_2,\xi_3,\ldots,\xi_n,w)|\intd v_1(\eta_2)\\
			&~~~~+d_2\log\frac{r_2+|\xi_2|}{r_2}+d_2+\log\frac{C}{c}\bigg\}\intd v_1(\eta_1)+d_1\log\frac{r_1+|\xi_1|}{r_1}+d_1+\log\frac{C}{c}\\
			=&\frac{1}{\pi^2r_1^2r_2^2}\int_{\Delta(0,r_1)\times\Delta(0,r_2)}\log|H(\eta_1,\eta_2,\xi_3,\ldots,\xi_n,w)|\intd v_2(\eta_1,\eta_2)\\
			&~~+d_2\log\frac{r_2+|\xi_2|}{r_2}+d_1\log\frac{r_1+|\xi_1|}{r_1}+d_2+d_1+2\log\frac{C}{c}\\
			&\ldots\\
			\leq&\frac{1}{\pi^nr_1^2\ldots r_n^2}\int_{P_{r_1,\ldots,r_n}}\log|H(\eta,w)|\intd v_n(\eta)+\sum_{i=1}^nd_i\log\frac{r_i+|\xi_i|}{r_i}+\sum_{i=1}^nd_i+n\log\frac{C}{c}\\
			=&\frac{1}{|P_{r_1,\ldots,r_n}|}\int_{P_{r_1,\ldots,r_n}}\log|H(\eta,w)|\intd v_n(\eta)+\sum_{i=1}^nd_i\log\frac{r_i+|\xi_i|}{r_i}+\sum_{i=1}^nd_i+n\log\frac{C}{c}.
		\end{flalign*}
		This completes the proof.
	\end{proof}

	\begin{lem}\label{lem: change basis}
		Suppose $h\in \Hol(U)$, where $U$ is an open neighborhood of $0\in\C^k$, $h$ is not identically zero. Then there exists an orthonormal basis $\{e_1,\ldots,e_k\}$ of $\C^k$, such that for every $i\in\{1,\cdots,k\}$, the one-variable function $h(\lambda e_i)$ is not identically zero.
	\end{lem}
	
	Lemma \ref{lem: change basis} can be implied by \cite[Lemma 2, Page 33]{chirka}. To avoid employing more terminologies we give a straightforward proof. We thank Dr. Hui Dan for suggesting this proof to us.
	
	\begin{proof}
		In the case $k=1$, the conclusion is obvious. In the case $k=2$, notice that $\langle(z_1,z_2),(\bar{z}_2,-\bar{z}_1)\rangle=0$ for all pairs $(z_1,z_2)$. Let
		$$
		f(z_1,z_2)=h(z_1,z_2)\overline{h(\bar{z}_2,-\bar{z}_1)}.
		$$
		The function $f$ is a holomorphic function defined in an open neighborhood of $0$ and $f$ is not identically zero. Pick any $(z_1,z_2)\neq0$ so that $f(z)\neq0$ and normalize $\{(z_1,z_2), (\bar{z}_2,-\bar{z}_1)\}$ into an orthonormal basis. This will satisfy our condition.
		
		We prove the general case by induction, suppose we have proved the result for $U\subseteq\C^{k-1}$. Now for $U\subset\C^k$, pick $z\neq0$ so that $h(z)\neq0$. Pick a two dimensional subspace $L\subset\C^k$ containing $z$, then $h|_{L\cap U}\nequiv0$. Since $\dim L=2$, by the previous argument we have orthonormal $v_1$ and $v_2\in L\cap U$ so that $h$ is not identically $0$ on $\C v_1\cap U$ and $\C v_2\cap U$. Now consider $L'=v_1^{\perp}$, since $v_2\in L'\cap U$, $h|_{L'\cap U}\nequiv0$. By induction, we have orthonormal $\{e_2,\ldots,e_k\}\subseteq L'$ such that $h$ is not identically $0$ on $\C e_i\cap U$, $i=2,\ldots,k$. The set $\{v_1, e_2,\ldots, e_k\}$ is the desired basis. This completes the proof.
	\end{proof}

	\begin{lem}\label{lem: key inequality local}
		Suppose $\Omega\subset\cn$ is a bounded strongly pseudoconvex domain with smooth boundary and $\rho$ is a defining function of $\Omega$. Suppose $h$ is a holomorphic function defined in an open neighborhood $U$ of $\overline{\Omega}$. Then for any $\zeta\in\partial\Omega$ there exist $C>0, N>0$, and an open neighborhood $V$ of $\zeta$, such that for any $z, w\in\Omega\cap V$ and $f\in\Hol(\Omega)$,
		\begin{equation}\label{eqn: key inequality local}
			|h(z)f(w)|\leq C\frac{F(z,w)^N}{|\rho(w)|^{N+n+1}}\int_{E(w,1)}|h(\lambda)f(\lambda)|\intd v_n(\lambda).
		\end{equation}
	\end{lem}
	
	\begin{proof}
		Without loss of generality, assume $\zeta=0$. Set
		\[
		e_{0,1}=\frac{\bpartial\rho(0)}{|\bpartial\rho(0)|},
		\]
		the unit normal vector at $0\in\partial\Omega$. Choose $e_{0,2}, \ldots, e_{0,n}$ so that 
		\[
		\bfe_0=\{e_{0,1}, e_{0,2},\ldots, e_{0,n}\}
		\]
		is an orthonormal basis of $\cn$.
			Recall that the complex tangent space at $0$ is
		\[
		T^{\C}_{0}(\partial\Omega)=\bigg\{\xi\in\cn: \sum_{j=1}^n\frac{\partial \rho(0)}{\partial z_j}\xi_j=0\bigg\}=\bigg(\C e_{0,1}\bigg)^\perp.
		\]
		Below let us work under the basis $\mathbf{e}_0$. Then
		\[
		T^{\C}_{0}(\partial\Omega)=\{\xi\in\cn: \xi_1=0\}.
		\]

		We split the proof into two cases.
		
		~
		
		\noindent\textbf{Case 1:} $h|_{\C e_{0,1}\cap U}$ is not identically zero.
		
		If $h|_{T^{\C}_0(\partial\Omega)\cap U}\equiv0$, then there exists a positive integer $d$ such that
		\[
		h(z)=z_1^dh_1(z),\quad\text{ where }h_1|_{\C e_{0,1}\cap U}\nequiv0,\quad h_1|_{T^{\C}_0(\partial\Omega)\cap U}\nequiv0.
		\]
		If $h|_{T^{\C}_0(\partial\Omega)\cap U}\nequiv 0$, then simply take $d=0$ and $h_1=h$. Thus in either case we have
		\begin{equation}\label{eqn: temp 8}
		h(z)=z_1^dh_1(z),\quad\text{ where }h_1|_{\C e_{0,1}\cap U}\nequiv0,\quad h_1|_{T^{\C}_0(\partial\Omega)\cap U}\nequiv0.
		\end{equation}
	By Lemma \ref{lem: W decomp}, we can modify our choice of $e_{0,2}, \ldots, e_{0,n}$ so that
	\begin{equation}\label{eqn: temp 9}
	h_1|_{\C e_{0,i}\cap U}\nequiv0, \quad \forall i=1,\ldots,n.
	\end{equation}
		For $w$ close to $0$, define
		\[
		e_{w,1}=\frac{\bpartial\rho(\pi(w))}{|\bpartial\rho(\pi(w))|},
		\]
		the unit normal vector at $w$.
		By Definition \ref{defn: pseudoconvex}, the vector $e_{w,1}$ varies smoothly with respect to $w$.
		Apply the Gram-Schmidt method to $\{e_{w,1}, e_{0,2}, \ldots, e_{0,n}\}$ to get an orthonormal basis $\bfe_w=\{e_{w,1}, e_{w,2},\ldots,e_{w,n}\}$.  Then the definition is consistent at $w=0$, and the vectors in $\bfe_w$ vary smoothly with respect to $w$ in a neighborhood of $0$. 
		Define the map
		\[
		U_w: \cn\to\cn, \quad\xi\mapsto w+\sum_{i=1}^n\xi_ie_{w,i}.
		\]
For $s, t>0$, denote
\[
P_w(s,t)=\{w+\sum_{i=1}^n\xi_i e_{w,i}: |\xi_1|<s, |\xi_j|<t, j=2, \ldots, n\}=U_w(P_{s,t,\ldots,t}).
\]
This is consistent with our notation in the preliminary.
By Lemma \ref{lem: Kball equiv to polydisk}, there exist $a, b>0$ such that for $w\in\Omega$ close enough to $0\in\partial\Omega$, 
\begin{equation*}
P_w(a\delta(w),b\delta(w)^{1/2})\subseteq E(w,1).
\end{equation*}
Let
		\[
		H(\xi, w)=h(U_w(\xi)),\quad H_1(\xi,w)=h_1(U_w(\xi)),
		\]
		and
		\[
		P(\xi,w)=\big[(U_w(\xi))_1\big]^d=\big[w_1+\sum_{i=1}^n\xi_ie_{w,j,1}\big]^d.
		\]
		Then the functions are defined in $\big(\Delta(0,R)\big)^n\times\big(\Delta(0,R)\big)^n$ for some $R>0$. By \eqref{eqn: temp 8}, $H=H_1P$. For each $w$, $P(\xi,w)$ is a polynomial of degree not exceeding $d$ in $\xi$. And
		$H_1$ is holomorphic in $\xi$. Also by \eqref{eqn: temp 9},
		\[
		H_1(\xi_i e_i,0)=h_1(\xi_i e_{w,i})\nequiv 0,\quad\forall i=1,\ldots,n.
		\]
		Let $d_i$ be the multiplicity of $H_1(\xi_i e_i,0)$ at $\xi_i=0$.
		By Lemma \ref{lem: W decomp} and Lemma \ref{lem: key inequality ln}, there exist $0<r<R$ and $C>0$ such that whenever $0<r_1, r_2,\ldots, r_n<r$, $\xi, w\in\big(\Delta(0,r)\big)^n$,
		\begin{equation*}
		\log|H_1(\xi,w)|\leq\frac{1}{|P_{r_1,\ldots,r_n}|}\int_{P_{r_1,\ldots,r_n}}\log|H_1(\eta,w)|\intd v_n(\eta)+\sum_{i=1}^nd_i\log\frac{r_i+|\xi_i|}{r_i}+\sum_{i=1}^nd_i+n\log\frac{C}{c}.
		\end{equation*}
		Also by Lemma \ref{lem: key inequality dim n degree d},
		\begin{equation*}
			\log|P(\xi,w)|\leq\frac{1}{|P_{r_1,\ldots,r_n}|}\int_{P_{r_1,\ldots,r_n}}\log|P(\eta,w)|\intd v_n(\eta)+d\sum_{i=1}^n\log\frac{r_i+|\xi_i|}{r_i}+nd.
		\end{equation*}
		Adding up the two inequalities above, and by $H=H_1P$, we get that for $0< r_1,\ldots,r_n<r$ and $\xi, w\in\big(\Delta(0,r)\big)^n$,
		\begin{equation*}
		\log|H(\xi,w)|\leq\frac{1}{|P_{r_1,\ldots,r_n}|}\int_{P_{r_1,\ldots,r_n}}\log|H(\eta,w)|\intd v_n(\eta)+\sum_{i=1}^n(d_i+d)\log\frac{r_i+|\xi_i|}{r_i}+C',
		\end{equation*}
		where $C'=\sum_{i=1}^nd_i+nd+n\log\frac{C}{c}$.
		Take
		\[
		r_1=a\delta(w), r_2=r_3=\ldots=r_n=b\delta(w)^{1/2}.
		\]
		For $w\in\big(\Delta(0,r)\big)^n$ close enough to $0\in\partial\Omega$ we have $r_i<r, i=1,\ldots,n$. Then the above becomes
		\begin{flalign}\label{eqn: temp 10}
		&\log|h(U_w(\xi))|
		=\log|H(\xi,w)|\nonumber\\
		\leq&\frac{1}{|P_{r_1,\ldots,r_n}|}\int_{P_{r_1,\ldots,r_n}}\log|H(\eta,w)|\intd v_n(\eta)+\sum_{i=1}^n(d_i+d)\log\frac{r_i+|\xi_i|}{r_i}+C'\nonumber\\
		=&\frac{1}{|P_{r_1,\ldots,r_n}|}\int_{P_{r_1,\ldots,r_n}}\log|h(U_w(\eta))|\intd v_n(\eta)+\sum_{i=1}^n(d_i+d)\log\frac{r_i+|\xi_i|}{r_i}+C'\\
		\xlongequal{\lambda=U_w(\eta)}&\frac{1}{|P_w(a\delta(w),b\delta(w)^{1/2})|}\int_{P_w(a\delta(w),b\delta(w)^{1/2})}\log|h(\lambda)|\intd v_n(\lambda)+\sum_{i=1}^n(d_i+d)\log\frac{r_i+|\xi_i|}{r_i}+C'.\nonumber
		\end{flalign}
		For $z$ close enough to $0$, there is $\xi\in\big(\Delta(0,r)\big)^n$ with $z=U_w(\xi)$,
		\begin{equation*}
		|\xi_1|=|\langle z-w, e_{w,1}\rangle|\approx|\langle z-w,\bpartial \rho(\pi(w))\rangle|,\quad |\xi_i|\leq|z-w|.
		\end{equation*}
	Thus by Lemma \ref{lem: X F approx},
	\begin{equation}\label{eqn: temp 12}
	\frac{r_1+|\xi_1|}{r_1}\lesssim\frac{a\delta(w)+|\langle z-w,\bpartial \rho(\pi(w))\rangle|}{a\delta(w)}\lesssim\frac{F(z,w)}{|\rho(w)|},
	\end{equation}
\begin{equation}\label{eqn: temp 13}
\frac{r_i+|\xi_i|}{r_i}\lesssim\frac{b\delta(w)^{1/2}+|z-w|}{b\delta(w)^{1/2}}\lesssim\bigg(\frac{F(z,w)}{|\rho(w)|}\bigg)^{1/2},\quad i=2,\ldots,n.
\end{equation}
Combining \eqref{eqn: temp 10}, \eqref{eqn: temp 12} and \eqref{eqn: temp 13} we see that for $z, w\in\Omega$ close to $0$,
\[
\log|h(z)|\leq\frac{1}{|P_w(a\delta(w),b\delta(w)^{1/2})|}\int_{P_w(a\delta(w),b\delta(w)^{1/2})}\log|h(\lambda)|\intd v_n(\lambda)+N\log\frac{F(z,w)}{|\rho(w)|}+C'',
\]
where $N=d_1+d+\frac{\sum_{i=2}^n(d_i+d)}{2}$, and $C''\in\mathbb{R}$ is some other constant.	

Suppose $f\in\Hol(\Omega)$, then $\log|f|$ is plurisubharmonic. So
\[
\log|f(w)|\leq\frac{1}{|P_w(a\delta(w),b\delta(w)^{1/2})|}\int_{P_w(a_1\delta(w),b_1\delta(w)^{1/2})}\log|f(\lambda)|\intd v_n(\lambda).
\]
Adding up the two inequalities above, we get
\begin{flalign*}
\log|f(w)h(z)|\leq&\frac{1}{|P_w(a\delta(w),b\delta(w)^{1/2})|}\int_{P_w(a\delta(w),b\delta(w)^{1/2})}\log|f(\lambda)h(\lambda)|\intd v_n(\lambda)\\
&~~~+N\log\frac{F(z,w)}{|\rho(w)|}+C''.
\end{flalign*}
Finally, by the Jensen's inequality,
\begin{flalign*}
|f(w)h(z)|\leq&\exp\bigg\{\frac{1}{|P_w(a\delta(w),b\delta(w)^{1/2})|}\int_{P_w(a\delta(w),b\delta(w)^{1/2})}\log|f(\lambda)h(\lambda)|\intd v_n(\lambda)\\
&~~~~~~~~~~~+N\log\frac{F(z,w)}{|\rho(w)|}+C''\bigg\}\\
\approx&\bigg(\frac{F(z,w)}{|\rho(w)|}\bigg)^N\exp\bigg\{\frac{1}{|P_w(a\delta(w),b\delta(w)^{1/2})|}\int_{P_w(a\delta(w),b\delta(w)^{1/2})}\log|f(\lambda)h(\lambda)|\intd v_n(\lambda)\bigg\}\\
\leq&\bigg(\frac{F(z,w)}{|\rho(w)|}\bigg)^N\frac{1}{|P_w(a\delta(w),b\delta(w)^{1/2})|}\int_{P_w(a\delta(w),b\delta(w)^{1/2})}|f(\lambda)h(\lambda)|\intd v_n(\lambda)\\
\lesssim&\frac{F(z,w)^N}{|\rho(w)|^{N+n+1}}\int_{P_w(a\delta(w),b\delta(w)^{1/2})}|f(\lambda)h(\lambda)|\intd v_n(\lambda)\\
\leq&\frac{F(z,w)^N}{|\rho(w)|^{N+n+1}}\int_{E(w,1)}|f(\lambda)h(\lambda)|\intd v_n(\lambda).
\end{flalign*}
This proves \eqref{eqn: key inequality local} at case 1.

		~
		
\noindent\textbf{Case 2:} $h|_{\C e_{\zeta,1}\cap U}\equiv0$.
		
We {\bf claim }that, there exists a biholomorphic map $\Phi$ defined in an open neighborhood of $\overline{\Omega}$, such that $h\circ\Phi^{-1}$ is not identically zero along the complex normal direction of $\Phi(\Omega)$ at the point $\Phi(\zeta)$.

To prove the claim, take a small open ball $B$ in $\Omega$ that is tangent to $\partial\Omega$ at $\zeta$. Let us temporarily drop the assumption that $\zeta=0$. Without loss of generality, assume instead that $B=\bn$, the open unit ball of $\cn$, and $\zeta=(1,0,\ldots,0)$. Clearly, $h|_{\bn}\nequiv0$. Choose $z\in\bn$ such that $h(z)\neq0$. Let $\varphi_z$ be the M\"{o}bius transform
\begin{equation}\label{eqn: Mobius}
\varphi_z(w)=\frac{z-P_z(w)-(1-|z|^2)^{1/2}Q_z(w)}{1-\langle w, z\rangle},
\end{equation}
where $P_z, Q_z$ are orthogonal projections from $\cn$ onto $\C z$ and $\big(\C z\big)^\perp$, respectively. 
It is well-known that $\varphi_z$ is biholomorphic on $\{w: \langle w, z\rangle\neq1\}$, $\varphi_z^{-1}=\varphi_z$, and $\varphi_z$ is an automorphism of $\bn$. If we choose $z$ close enough to $0$ then $\varphi_z$ is defined in an open neighborhood of $\overline{\Omega}$. Since $\partial\Omega$ and $\partial\bn$ are tangent at $\zeta$, the complex normal directions of $\partial\big(\varphi_z\Omega\big)$ and $\partial\big(\varphi_z\bn\big)=\partial\bn$ coincide at $\varphi_z(\zeta)$. The complex normal direction of $\partial\bn$ at $\varphi_z(\zeta)$ is represented by $\C \varphi_z(\zeta)$. By our construction, $h\circ\varphi_z^{-1}(0)=h(z)\neq0$. Therefore
\[
h\circ\varphi_z^{-1}|_{\C \varphi_z(\zeta)}\nequiv0.
\]
In other words, if we take $\Phi=\varphi_z$ then $h\circ\Phi^{-1}$ is not identically zero in the complex normal direction at $\Phi(\zeta)$. This proves the claim.

Let $\Phi$ be as in the claim. Then $\Phi(\Omega)$ is a bounded strongly pseudoconvex domain with smooth boundary and $\rho\circ\Phi^{-1}$ is a defining function for $\Phi(\Omega)$. By our claim and case 1, there exist $N>0$, $C>0$ and an open neighborhood $U$ of $\Phi(\zeta)$ such that for $\xi, \omega\in U$ and $g\in\Hol(\Phi(\Omega))$,
\[
|h\circ\Phi^{-1}(\xi)g(\omega)|\leq C\frac{F'(\xi,\omega)^N}{|\rho\circ\Phi^{-1}(\omega)|^{N+n+1}}\int_{E'(\omega,1)}|h\circ\Phi^{-1}(\eta)g(\eta)|\intd v_n(\eta).
\]
Here $F'(\xi, \omega)$ defined as in Definition \ref{defn: X F}, for $\Phi(\Omega)$, and $E'(\omega,1)$ is the unit Kobayashi ball for $\Phi(\Omega)$ centered at $\omega$. Since the Kobayashi distance is invariant under a biholomorphic map (\cite{krantz}), we have
\[
\Phi^{-1}E'(\omega,1)=E(\Phi^{-1}(\omega),1),\quad\forall\omega\in\Phi(\Omega).
\]

 Since $\Phi$ is biholomorphic, $V=\Phi^{-1}(U)$ is an open neighborhood of $\zeta$. For $z, w\in V$ there are $\xi, \omega\in U$ such that $\xi=\Phi(z), \omega=\Phi(w)$. For $f\in\Hol(\Omega)$, $g=f\circ\Phi^{-1}\in\Hol(\Phi(\Omega))$. Therefore, we have
\begin{flalign*}
|h(z)f(w)|=&|h\circ\Phi^{-1}(\xi)g(\omega)|\\
\leq&C\frac{F'(\xi,\omega)^N}{|\rho\circ\Phi^{-1}(\omega)|^{N+n+1}}\int_{E'(\omega,1)}|h\circ\Phi^{-1}(\eta)g(\eta)|\intd v_n(\eta)\\
\xlongequal{\lambda=\Phi^{-1}(\eta)}&C\frac{F'(\Phi(z),\Phi(w))^N}{|\rho(w)|^{N+n+1}}\int_{E(w,1)}|h(\lambda)f(\lambda)|J(\lambda)\intd v_n(\lambda)\\
\lesssim&\frac{F(z,w)^N}{|\rho(w)|^{N+n+1}}\int_{E(w,1)}|h(\lambda)f(\lambda)|\intd v_n(\lambda).
\end{flalign*}
Here $J(\lambda)$ is the real Jacobian of $\Phi$. The last inequality comes from Lemma \ref{lem: biholomorphic} and the fact that $J(\lambda)$ is bounded on $\overline{\Omega}$. This completes the proof.
	\end{proof}

\begin{proof}[{\bf Proof of Lemma \ref{lem: key inequality Kdelta}}]
Since $\partial\Omega$ is compact, by Lemma \ref{lem: key inequality local}, there is a finite open cover $\{V_i\}_{i=1}^k$ of $\partial\Omega$, and constants $C_i, N_i>0, i=1, \ldots, k$ such that for any $f\in\Hol(\Omega)$, $ i=1,\ldots,k,$ and $z, w\in V_i\cap\Omega,$
\[
|h(z)f(w)|\leq C_i\frac{F(z,w)^{N_i}}{|\rho(w)|^{N_i+n+1}}\int_{E(w,1)}|h(\lambda)f(\lambda)|\intd v_n(\lambda).
\]
By Definition \ref{defn: X F}, $\frac{F(z,w)}{|\rho(w)|}\geq1$. Take $N=\max\{N_1,\ldots, N_k\}$. Then the above inequality implies
\[
|h(z)f(w)|\lesssim\frac{F(z,w)^{N}}{|\rho(w)|^{N+n+1}}\int_{E(w,1)}|h(\lambda)f(\lambda)|\intd v_n(\lambda)
\]
whenever $(z,w)\in\big(\cup_iV_i\times V_i\big)\cap\big(\Omega\times\Omega\big)$.
It is easy to see that $R_\delta\subset\cup_iV_i\times V_i$ for some $\delta>0$. So the inequality holds for $R_\delta\cap\big(\Omega\times\Omega\big)\subset\big(\cup_iV_i\times V_i\big)\cap\big(\Omega\times\Omega\big)$. This completes the proof.
\end{proof}

We are ready to prove the key inequality.
\begin{proof}[{\bf Proof of Theorem \ref{thm: key inequality}}]
Let $\delta, N, C$ be as in Lemma \ref{lem: key inequality Kdelta}. It suffices to prove \eqref{eqn: key inequality} for $(z,w)\notin R_\delta$. We first prove
\begin{equation}\label{eqn: temp 15}
|f(w)|\lesssim\frac{1}{|\rho(w)|^{N+n+1}}\int_{E(w,1)}|h(\lambda)f(\lambda)|\intd v_n(\lambda),\quad\forall w\in\Omega.
\end{equation}
Denote $\bn(z,r)$ the Euclidean open ball in $\cn$ with center $z$ and radius $r$.
Take a finite set 
\[
\{z_1, z_2,\ldots, z_k\}\subset\Omega_{\delta/4}
\]
such that
\[
h(z_i)\neq0,\quad \forall i=1, 2, \ldots, k,
\]
and
\[
\bigcup_{i=1}^k\bn(z_i,\frac{\delta}{3})\supset\Omega_{\delta/4}.
\]
Then for any $w\in\Omega_{\delta/4}$, there is some $z_i$ such that $|z_i-w|<\delta/3$. So
\[
\delta(z_i)+\delta(w)+|z-w|<\delta/4+\delta/4 +\delta/3<\delta,
\]
i.e., $(z_i, w)\in R_\delta$. Thus by Lemma \ref{lem: key inequality Kdelta},
\[
|h(z_i)f(w)|\leq C\frac{F(z_i,w)^N}{|\rho(w)|^{N+n+1}}\int_{E(w,1)}|h(\lambda)f(\lambda)|\intd v_n(\lambda).
\]
Denote
\[
\delta'=\min\{|\rho(z_j)|: j=1,\ldots,k\},\quad m=\min\{|h(z_j)|: j=1,\ldots,k\}.
\]
Then $\delta'>0, m>0$ by construction.
By Definition \ref{defn: X F},
\[
F(z_i, w)\geq\delta'.
\]
So we have
\begin{flalign*}
|f(w)|\leq&\frac{|h(z_i)f(w)|}{m}\\
\leq&\frac{1}{m} C\frac{F(z_i,w)^N}{|\rho(w)|^{N+n+1}}\int_{E(w,1)}|h(\lambda)f(\lambda)|\intd v_n(\lambda)\\
\leq&\frac{\delta'^{N}}{m} \frac{1}{|\rho(w)|^{N+n+1}}\int_{E(w,1)}|h(\lambda)f(\lambda)|\intd v_n(\lambda).
\end{flalign*}
This proves \eqref{eqn: temp 15} when $w\in\Omega_{\delta/4}$.
On the other hand, there is $\epsilon>0$ such that
\[
|\rho(w)|>\epsilon,\quad E(w,1)\supset \bn(w,\epsilon),\quad\forall w\in\Omega\backslash\Omega_{\delta/4}.
\]
By \cite[Proposition 2.1.10]{krantz}, the subharmonic function $\log|h|$ is locally integrable, and therefore the function $w\mapsto\frac{1}{|\bn(w,\epsilon)|}\int_{\bn(w,\epsilon)}\log|h(\lambda)|\intd v_n(\lambda)$ is continuous. This proves that there is some $c\in\mathbb{R}$ such that
\begin{equation*}
c\leq\frac{1}{|\bn(w,\epsilon)|}\int_{\bn(w,\epsilon)}\log|h(\lambda)|\intd v_n(\lambda),\quad \forall w\in\Omega\backslash\Omega_{\delta/4}.
\end{equation*}
Also,
\begin{equation*}
\log|f(w)|\leq\frac{1}{|\bn(w,\epsilon)|}\int_{\bn(w,\epsilon)}\log|f(\lambda)|\intd v_n(\lambda).
\end{equation*}
Adding up the two inequalities above we get
\begin{equation*}
c+\log|f(w)|\leq\frac{1}{|\bn(w,\epsilon)|}\int_{\bn(w,\epsilon)}\log|f(\lambda)h(\lambda)|\intd v_n(\lambda),\quad \forall w\in\Omega\backslash\Omega_{\delta/4}.
\end{equation*}
By Jensen's inequality,
\begin{flalign*}
e^c|f(w)|\leq&\exp\bigg\{\frac{1}{|\bn(w,\epsilon)|}\int_{\bn(w,\epsilon)}\log|f(\lambda)h(\lambda)|\intd v_n(\lambda)\bigg\}\\
\leq&\frac{1}{|\bn(w,\epsilon)|}\int_{\bn(w,\epsilon)}|f(\lambda)h(\lambda)|\intd v_n(\lambda)\\
\lesssim&\int_{\bn(w,\epsilon)}|f(\lambda)h(\lambda)|\intd v_n(\lambda)\\
\leq&\int_{E(w,1)}|f(\lambda)h(\lambda)|\intd v_n(\lambda)\\
\lesssim&\frac{1}{|\rho(w)|^{N+n+1}}\int_{E(w,1)}|f(\lambda)h(\lambda)|\intd v_n(\lambda),\quad \forall w\in\Omega\backslash\Omega_{\delta/4}.
\end{flalign*}
This proves \eqref{eqn: temp 15} for $w\in\Omega\backslash\Omega_{\delta/4}$, and hence \eqref{eqn: temp 15} holds for any $w\in\Omega$.
If $(z,w)\notin R_\delta$, then $F(z,w)$ is bounded away from $0$. Also, $|h|$ has an upper bounded. So
\begin{flalign*}
|h(z)f(w)|\lesssim& |f(w)|\\
\lesssim&\frac{1}{|\rho(w)|^{N+n+1}}\int_{E(w,1)}|h(\lambda)f(\lambda)|\intd v_n(\lambda)\\
\lesssim&\frac{F(z,w)^N}{|\rho(w)|^{N+n+1}}\int_{E(w,1)}|h(\lambda)f(\lambda)|\intd v_n(\lambda).
\end{flalign*}
Combing Lemma \ref{lem: key inequality Kdelta} we see that the inequality above holds for any $(z,w)\in\Omega\times\Omega$. This completes the proof.
\end{proof}
	
	\section{Main Result}\label{sec: main}
	In this section, we prove the following theorem, which is the main result of this paper.
	\begin{thm}\label{thm: main}
		Suppose $\Omega\subseteq\cn$ is a bounded strongly pseudoconvex domain with smooth boundary and $h\in \Hol(\overline{\Omega})$. Then the principal submodule of the Bergman module $L_a^2(\Omega)$ generated by $h$ is $p$-essentially normal for all $p>n$.
	\end{thm}
The main ingredient of the proof of Theorem \ref{thm: main} is Theorem \ref{thm: key inequality}. Besides that, some other preparations are needed.

For a measurable function $G$ on $\Omega\times\Omega$, formally define the integral operator
\begin{equation}\label{eqn: A_G}
A_Gf(z)=\int_{\Omega}\frac{G(z,w)}{F(z,w)^{n+1}}f(w)\intd v_n(w).
\end{equation}

\begin{lem}\label{lem: integral bounded}
Suppose $G(z,w)\in L^\infty(\Omega\times\Omega)$. Then $A_G$ defines a bounded operator on $L^2(\Omega)$.
\end{lem}

\begin{proof}
By Lemma \ref{lem: Rudin Forelli},
\[
\int_{\Omega}\bigg|\frac{G(z,w)}{F(z,w)^{n+1}}\bigg||\rho(w)|^{-1/2}\intd v_n(w)\lesssim\int_\Omega\frac{|\rho(w)|^{-1/2}}{F(z,w)^{n+1}}\intd v_n(w)\lesssim |\rho(z)|^{-1/2},
\]
and
\[
\int_\Omega\bigg|\frac{G(z,w)}{F(z,w)^{n+1}}\bigg||\rho(z)|^{-1/2}\intd v_n(z)\lesssim\int_\Omega\frac{|\rho(z)|^{-1/2}}{F(z,w)^{n+1}}\intd v_n(z)\lesssim |\rho(w)|^{-1/2}.
\]
By Schur's test, $A_G$ is bounded. 
\end{proof}

As in \cite[Lemma 5.1]{paper1}, by Lemma \ref{lem: integral bounded} and interpolation, we have the following Schatten-class membership criterion.
	\begin{lem}\label{lem: Schatten criterion}
		Suppose $2\leq p<\infty$ and $G(z,w)$ is a measurable function in $\Omega\times \Omega$. Let $A_G$ be as in \eqref{eqn: A_G}.
		If
		$$
		\int_{\Omega}\int_{\Omega}\frac{|G(z,w)|^p}{F(z,w)^{2(n+1)}}dv_n(z)dv_n(w)<\infty,
		$$
		then the operator $A_G$ is in the Schatten-$p$ class $\mathcal{S}^p$.
	\end{lem}

An immediate consequence of Lemma \ref{lem: Schatten criterion} is the following.
\begin{lem}\label{lem: Mrho Schatten}
Suppose $c>0$. Denote $P$ the orthogonal projection from $L^2(\Omega)$ onto $L_a^2(\Omega)$. Then for $ p>\frac{n}{c}, p\geq2$ it hods that
\[
 M_{|\rho|^c}P\in\mathcal{S}^p.
\]
\end{lem}
\begin{proof}
For any $f\in L^2(\Omega)$,
\begin{equation*}
M_{|\rho|^c}Pf(z)=\int_{\Omega}|\rho(z)|^cK(z,w)f(w)\intd v_n(w)=\int_{\Omega}\frac{|\rho(z)|^cK(z,w)F(z,w)^{n+1}}{F(z,w)^{n+1}}f(w)\intd v_n(w).
\end{equation*}
For $p>\frac{n}{c}$, $p\geq2$, by \eqref{eqn: Kzw lesssim Fzw} and Lemma \ref{lem: Rudin Forelli},
\begin{flalign*}
&\int_{\Omega}\int_{\Omega}\frac{\big||\rho(z)|^cK(z,w)F(z,w)^{n+1}\big|^p}{F(z,w)^{2(n+1)}}\intd v_n(z)\intd v_n(w)\\
\lesssim&\int_{\Omega}\int_{\Omega}\frac{|\rho(z)|^{cp}}{F(z,w)^{2(n+1)}}\intd v_n(z)\intd v_n(w)\\
\lesssim&\int_{\Omega}|\rho(w)|^{-(n+1-cp)}\intd v_n(w).
\end{flalign*}
Since $-(n+1-cp)>-n-1+n=-1$, by Lemma \ref{lem: Rudin Forelli} again, the integral above is finite. Therefore by Lemma \ref{lem: Schatten criterion}, $M_{|\rho|^c}P\in\mathcal{S}^p$. This completes the proof.
\end{proof}

For $l\geq0$, let $G_l(z,w)$ be as in Lemma \ref{lem: weighted Bergman kernel}. For $i=1,\ldots, n$, formally define the integral operator,
\begin{equation}\label{eqn: Gil defn}
G_i^{(l)}f(z)=\int_\Omega\bar{w}_if(w)G_l(z,w)|\rho(w)|^l\intd v_n(w).
\end{equation}

\begin{lem}\label{lem: Gi bounded}
For $l\geq0, i=1,\ldots,n$, $G_i^{(l)}$ extends to a bounded operator from $L^2(\Omega)$ to $L_a^2(\Omega)$. 
\end{lem}

\begin{proof}
	By Definition \ref{defn: X F} and \eqref{eqn: Glzw lesssim Fzw},
	\[
	\big|\bar{w}_iG_l(z,w)\big||\rho(w)|^l\lesssim \frac{|\rho(w)|^l}{F(z,w)^{n+1+l}}\leq\frac{1}{F(z,w)^{n+1}}.
	\]
	By Lemma \ref{lem: integral bounded}, $G_i^{(l)}$ is bounded on $L^2(\Omega)$. 
Since $G_l(z,w)$ is holomorphic in $z$, for every $f\in L^2(\Omega)$, its image $G_i^{(l)}f$ is a holomorphic function. Therefore $G_i^{(l)}$ has range in $ L_a^2(\Omega)$. This completes the proof.
\end{proof}

	\begin{proof}[\textbf{Proof of Theorem \ref{thm: main}}]
	Denote $[h]$ the principal submodule of $L_a^2(\Omega)$ generated by $h$. 
	\[
	[h]=\overline{\{hf: f\in L_a^2(\Omega)\}}\subset L^2(\Omega).
	\]
	In this proof we consider $[h]$ as a Hilbert subspace of $L^2(\Omega)$. All operators on $[h]$ can be identified with its extension to $L^2(\Omega)$ that take zero on $L^2(\Omega)\ominus[h]$. 
	Denote $P_h$ the orthogonal projection from $L^2(\Omega)$ onto $[h]$. For $i=1,\ldots, n$, the module action on $[h]$ is defined by the operators
	\[
	R_i=P_hM_{z_i}P_h,\quad i=1, 2,\ldots, n,
	\]
	where $M_{z_i}$ is the multiplication operator on $L^2(\Omega)$. Since
	\begin{flalign}\label{eqn: RiRj}
	[R_i, R_j^*]=&P_hM_{z_i}P_hM_{\bar{z}_j}P_h-P_hM_{\bar{z}_j}P_hM_{z_i}P_h\nonumber\\
	=&P_hM_{\bar{z}_j}(I-P_h)M_{z_i}P_h-P_hM_{z_i}(I-P_h)M_{\bar{z}_j}P_h\\
	=&-\bigg[(I-P_h)M_{\bar{z}_i}P_h\bigg]^*\bigg[(I-P_h)M_{\bar{z}_j}P_h\bigg],\nonumber
	\end{flalign}
it suffices to prove that for any $i=1, 2, \ldots, n$,
\begin{equation}\label{eqn: temp goal}
(I-P_h)M_{\bar{z}_i}P_h\in\mathcal{S}^p,\quad\forall p>2n.
\end{equation}
Let $N>0$ be as in Theorem \ref{thm: key inequality}. Take a positive integer $l>N+1$. For $i=1,\ldots,n$, let $G_i^{(l)}$ be as in \eqref{eqn: Gil defn}. Define
\begin{equation}
X_i(hf)=M_{\bar{z}_i}(hf)-hG_i^{(l)}(f),\quad\forall f\in L_a^2(\Omega).
\end{equation}
By Lemma \ref{lem: Gi bounded}, $G_i^{(l)}(f)\in L_a^2(\Omega)$. So $hG_i^{(l)}(f)\in[h]$. Therefore
\begin{equation}\label{eqn: temp}
(I-P_h)M_{\bar{z}_i}(hf)=(I-P_h)\bigg(M_{\bar{z}_i}(hf)-hG_i^{(l)}(f)\bigg)=(I-P_h)X_i(hf).
\end{equation}

For any $f\in L_a^2(\Omega), z\in\Omega$, since $f\in L_a^2(\Omega)\subset L_{a,l}^2(\Omega)$,
\begin{flalign*}
	X_i(hf)(z)=&\bar{z}_ih(z)\int_{\Omega}f(w)G_l(z,w)|\rho(w)|^l\intd v_n(w)-h(z)\int_{\Omega}\bar{w}_if(w)G_l(z,w)|\rho(w)|^l\intd v_n(w)\\
	=&\int_{\Omega}\overline{(z_i-w_i)}h(z)f(w)G_l(z,w)|\rho(w)|^l\intd v_n(w).
\end{flalign*}
By Definition \ref{defn: X F},
\[
|z-w|\leq F(z,w)^{1/2},\quad \forall z, w\in\Omega.
\]
By \eqref{eqn: Glzw lesssim Fzw},
\[
|G_l(z,w)|\lesssim F(z,w)^{-n-1-l},\quad \forall z, w\in\Omega.
\]
So by Theorem \ref{thm: key inequality} and the above,
\begin{flalign*}
	\big|X_i(hf)(z)\big|\lesssim&\int_{\Omega}|z-w|~|h(z)f(w)|~|G_l(z,w)|~ |\rho(w)|^l\intd v_n(w)\\
	\lesssim&\int_{\Omega}|h(z)f(w)|\frac{|\rho(w)|^l}{F(z,w)^{n+1/2+l}}\intd v_n(w)\\
	\lesssim&\int_{\Omega}\bigg(\frac{F(z,w)^N}{|\rho(w)|^{N+n+1}}\int_{E(w,1)}|h(\lambda)f(\lambda)|\intd v_n(\lambda)\bigg)\frac{|\rho(w)|^l}{F(z,w)^{n+1/2+l}}\intd v_n(w)\\
	=&\int_{\Omega}\int_{E(w,1)}|h(\lambda)f(\lambda)|\frac{|\rho(w)|^{l-N-(n+1)}}{F(z,w)^{n+1/2+(l-N)}}\intd v_n(\lambda)\intd v_n(w)\\
	=&\int_{\Omega}\bigg(\int_{E(\lambda,1)}\frac{|\rho(w)|^{l-N-(n+1)}}{F(z,w)^{n+1/2+(l-N)}}\intd v_n(w)\bigg)|h(\lambda)f(\lambda)|\intd v_n(\lambda).
\end{flalign*}
By Lemma \ref{lem: Kball equiv to polydisk}, Lemma \ref{lem: rzrw beta leq r} and Lemma \ref{lem: Fzw beta leq r}, we have
\begin{flalign*}
	&\int_{E(\lambda,1)}\frac{|\rho(w)|^{l-N-(n+1)}}{F(z,w)^{n+1/2+(l-N)}}\intd v_n(w)
	\approx\int_{E(\lambda,1)}\frac{|\rho(\lambda)|^{l-N-(n+1)}}{F(z,\lambda)^{n+1/2+(l-N)}}\intd v_n(w)\\
	=&\frac{|\rho(\lambda)|^{l-N-(n+1)}}{F(z,\lambda)^{n+1/2+(l-N)}}|E(\lambda,1)|
	\approx\frac{|\rho(\lambda)|^{l-N}}{F(z,\lambda)^{n+1/2+(l-N)}}.
\end{flalign*}
Combining the above shows that
\begin{equation}\label{eqn: Xi estimate}
	|X_i(hf)(z)|\lesssim\int_{\Omega}|h(\lambda)f(\lambda)|\frac{|\rho(\lambda)|^a}{F(z,\lambda)^{n+\frac{1}{2}+a}}\intd v_n(\lambda),
\end{equation}
where $a=l-N>1$. Define $Y_i$ on the linear subspace $L=\{|\rho|^{1/2}hf: f\in L_a^2(\Omega)\}\subset L^2(\Omega)$,
\[
Y_i(|\rho|^{1/2}hf)=X_i(hf).
\]
Then by \eqref{eqn: Xi estimate}, for $f\in L_a^2(\Omega)$,
\begin{flalign*}
|Y_i(|\rho|^{1/2}hf)(z)|\lesssim&\int_{\Omega}\big||\rho(\lambda)|^{1/2}h(\lambda)f(\lambda)\big|\frac{|\rho(\lambda)|^{a-1/2}}{F(z,\lambda)^{n+\frac{1}{2}+a}}\intd v_n(\lambda)\\
\lesssim&\int_{\Omega}\big|\big(|\rho|^{1/2}hf\big)(\lambda)\big|\frac{1}{F(z,\lambda)^{n+1}}\intd v_n(\lambda).
\end{flalign*}
By Lemma \ref{lem: integral bounded}, $Y_i$ extends to a bounded operator on $L^2(\Omega)$. By \eqref{eqn: temp},
\[
(I-P_h)M_{\bar{z}_i}P_h(hf)=(I-P_h)X_i(hf)=(I-P_h)Y_iM_{|\rho|^{1/2}}PP_h(hf).
\]
Since $\{hf: f\in L_a^2(\Omega)\}$ is dense in $[h]$, we have
\[
(I-P_h)M_{\bar{z}_i}P_h=(I-P_h)Y_iM_{|\rho|^{1/2}}PP_h.
\]
By Lemma \ref{lem: Mrho Schatten}, $M_{|\rho|^{1/2}}P\in\mathcal{S}^p, \forall p>2n$. So $(I-P_h)M_{\bar{z}_i}P_h\in\mathcal{S}^p, \forall p>2n$, for any $i=1,\ldots, n$. By \eqref{eqn: RiRj},
\[
[R_i, R_j^*]\in\mathcal{S}^p,\quad \forall p>n.
\]
This completes the proof.
\end{proof}

\section{The Geometric Arveson-Douglas Conjecture}

Suppose $V$ is a complex variety in $\Omega$. Define
\[
\PP_V=\{f\in L_a^2(\Omega): f|_V=0\}.
\]
Since evaluations are bounded on $L_a^2(\Omega)$, $\PP_V$ is a closed subspace of $L_a^2(\Omega)$. Clearly $\PP_V$ is closed under each $M_{z_i}$. So $\PP_V$ is a submodule of $L_a^2(\Omega)$.

In \cite{KS2012}, Kennedy and Shalit introduced a geometric version of the Arveson-Douglas Conjecture, which states that if $\Omega$ is $\bn$ and $V$ is homogeneous, then the quotient module $\mathcal{Q}_V:=\PP_V^\perp$ is $p$-essentially normal for any $p>\dim_{\C}V$. It quickly drew attentions of many researchers, see for example, \cite{DTY} and \cite{ee}. This geometric conjecture naturally extends to the setting of strongly pseudoconvex domains. As a byproduct of our Theorem \ref{thm: main} we give some result on the Geometric Arveson-Douglas Conjecture.

Let us begin with a few definitions. See \cite{chirka} for more details.

\begin{defn}
	Let $\Omega$ be a complex manifold. A set $A\subseteq\Omega$ is called a \emph{(complex) analytic subset} of $\Omega$ if for each point $a\in\Omega$ there are a neighborhood $U$ of $a$ and functions $f_1,\ldots,f_N$ holomorphic in this neighborhood such that
	$$
	A\cap U=\{z\in U:~f_1(z)=\cdots=f_N(z)=0\}.
	$$
	
	In particular, $A$ is said to be \emph{principal} if there is a holomorphic function $f$ on $\Omega$, not identically vanishing on any component of $\Omega$, such that $A=Z(f):=\{z\in\Omega:~f(z)=0\}$. The function $f$ is called a defining function of $A$ (not to be confused with the defining function of a pseudoconvex domain).
\end{defn}

For an analytic subset $A$ of $\Omega$, a point $a\in A$ is called a \emph{regular point} if there is a neighborhood $U$ of $a$ in $\Omega$ such that $A\cap U$ is a complex submanifold of $\Omega$. Otherwise $a$ is called a \emph{singular point}. The complex dimension of $A$ at a regular point $z$, denoted by $\dim_zA$, is simply the dimension of the complex manifold $A\cap U$, where $U$ is a sufficiently small neighborhood of $z$. The set of regular points is dense in $A$ (\cite{chirka}) and this leads to a definition of dimension at any point.

\begin{defn}
	Let $A$ be an analytic subset of $\Omega$. The \emph{dimension} of $A$ at an arbitrary point $a\in A$ is the number
	$$
	\dim_a A:=\limsup_{z\to a, z\text{ regular}}\dim_z A.
	$$
	The dimension of $A$ is, by definition, the maximum of its dimensions at all points:
	$$
	\dim A:=\max_{z\in A}\dim_z A.
	$$
	
	$A$ is said to be \emph{pure} if its dimensions at all points coincide.
\end{defn}

Pure analytic subsets of codimension $1$ has some very important properties.

\begin{lem}{\cite[Corollary 1, Page 26]{chirka}}\label{locally principal}
	Every pure $(n-1)$ dimensional analytic subset on an $n$-dimensional complex manifold is locally principal, i.e., for any $a\in A$ there exist open neighborhood $U$ of $a$ in $\Omega$ and holomorphic function $f$ on $U$ such that $A\cap U=\{z\in U:~f(z)=0\}$.
\end{lem}

Let $A$ be a principal analytic subset of $\Omega$, i.e., $A=\{z\in\Omega:~f(z)=0\}$ for a ceratin holomorphic function $f$. The function $f$ is called a \emph{minimal defining function} of $A$ if for every open set $U\subseteq\Omega$ and every $g\in \Hol(U)$ such that $g|_{A\cap U}=0$, there exists an $h\in \Hol(U)$ such that $g=fh$ in $U$.

\begin{lem}{\cite[Proposition 1, Page 27]{chirka}}\label{minimal defn func}
	Every pure $(n-1)$-dimensional analytic subset on an $n$-dimensional complex manifold locally has a minimal defining function.
\end{lem}

Now suppose $\Omega$ is a bounded strongly pseudoconvex domain with smooth boundary, $\rho$ is a defining function and $V$ is a pure $(n-1)$-dimensional analytic subset of an open neighborhood of $\overline{\Omega}$. For $\epsilon>0$, denote
\[
\Omega^\epsilon=\{z\in\cn: \rho(z)<\epsilon\}.
\] 
Choose $\epsilon$ sufficiently small so that $\Omega^{\epsilon}$ is also a bounded strongly pseudoconvex domain with smooth boundary, and $V$ is an analytic subset of $\Omega^{\epsilon}$. Since the gradient of $\rho$ is non-vanishing on $\partial\Omega$, we can take $\epsilon$ small enough so that $\overline{\Omega^{\epsilon/2}}$ is a compact subset of $\Omega^{\epsilon}$. By Lemma \ref{minimal defn func}, there is a finite open cover $\{U_i\}$ of $\overline{\Omega^{\epsilon/2}}\subset\Omega^{\epsilon}$ and a minimal defining function $h_i$ on each $U_i$. By definition, if $U_i\cap U_j\neq\emptyset$, the function $g_{ij}=h_i/h_j$ is holomorphic and non-vanishing on $U_i\cap U_j$. They satisfy
$$
g_{ij}\cdot g_{ji}=1 \text{ on } U_i\cap U_j,
$$
and
$$
g_{ij}\cdot g_{jk}\cdot g_{ki}=1 \text{ on }U_i\cap U_j\cap U_k.
$$
Such a set of functions is called a second Cousin data. The second Cousin problem asks for non-vanishing holomorphic functions $\{g_i: U_i\to\C\}$, such that $g_{ij}=g_i/g_j$ on each $U_i\cap U_j$. In the famous paper \cite{Oka}, Oka proved that the second Cousin problem on a domain of holomorphy can be solved by holomorphic functions if it can be solved by continuous functions. More generally, the following holds.

\begin{lem}[\cite{krantz} Corollary 6.5.15]\label{lem: second cousin}
Let $\Omega\subset\cn$ be a domain. If $H^1(\Omega,\mathscr{O})=H^2(\Omega,\mathbb{Z})=0$ then the second Cousin problem is always solvable on $\Omega$.
\end{lem}
If $\Omega$ is a bounded strongly pseudoconvex domain with smooth boundary then $H^1(\Omega,\mathscr{O})=0$. Therefore the second Cousin problem is solvable on a bounded strongly pseudoconvex domain $\Omega$ with smooth boundary satisfying the topological condition $H^2(\Omega,\mathbb{Z})=0$.

\begin{lem}\label{lem: P_V for contractible}
Suppose $\Omega\subset\cn$ is a bounded strongly pseudoconvex domain with smooth boundary and $H^2(\Omega^\epsilon,\mathbb{Z})=0$ for sufficiently small $\epsilon>0$. Suppose $V$ is a complex analytic subset of an open neighborhood of $\overline{\Omega}$ of pure dimension $n-1$. Then there is a holomorphic function $h$ defined in an open neighborhood of $\overline{\Omega}$ such that
\[
\PP_V=\{hf: f\in\Hol(\Omega), hf\in L_a^2(\Omega)\}.
\]
\end{lem}

\begin{proof}
By Lemma \ref{lem: second cousin} and the analysis above it, there exist non-vanishing holomorphic functions $g_i\in \Hol(U_i\cap \Omega^\epsilon)$ such that $g_{ij}=g_i/g_j$. Thus $h_i/h_j=g_i/g_j, \forall i, j.$ If we define $h=h_i/g_i$ on $U_i\cap\Omega^\epsilon$, then $h$ is well defined and becomes a global minimal defining function for $V$ in $\Omega^{\epsilon/2}$.
Suppose $f\in \PP_V$. Then $f=gh$ for some $g\in \Hol(\Omega)$. Thus
\[
\PP_V\subset\{gh: g\in\Hol(\Omega), gh\in L_a^2(\Omega)\}.
\]
The other side of inclusion is trivial by the definition of $\PP_V$. This completes the proof.
\end{proof}

Let $h$ be as in Lemma \ref{lem: P_V for contractible}. Then clearly
\[
[h]\subset \PP_V=\{hf: f\in\Hol(\Omega), hf\in L_a^2(\Omega)\}.
\]
If one can show that $[h]=\PP_V$, then Theorem \ref{thm: main} actually implies the $p$-essential normality of $\PP_V$ (and hence $\PP_V^\perp$) for $p>n$. In the following subsection, we show that this is indeed the case when $\Omega=\bn$.

\subsection{Characterization of $[h]$ on the Unit Ball}

As explained above, the main goal of this subsection is to prove the following.
\begin{thm}\label{thm: codim 1 on bn}
	Suppose $V$ is a pure $(n-1)$-dimensional analytic subset of an open neighborhood of $\clb$, then $V$ has a minimal defining function $h$ on an open neighborhood of $\clb$. Moreover,
	\begin{equation}\label{eqn: P_V=P_h=set}
	[h]=\PP_V=\{hf: f\in\Hol(\bn), hf\in L_a^2(\bn)\}.
	\end{equation}
In particular, the submodule $\PP_V$ and the quotient module $\PP_V^\perp$ are $p$-essentially normal for all $p>n$.
\end{thm}

In the case $\Omega=\bn$, one can take the defining function $\rho(z)=|z|^2-1$. Moreover,
\[
F(z,w)\approx|1-\langle z, w\rangle|,\quad\forall z, w\in\bn.
\]
The Kobayashi metric on $\bn$ is exactly the Bergman metric, given by
\[
\beta(z,w)=\tanh^{-1}|\varphi_z(w)|,
\]
where $\varphi_z(w)$ is the M\"{o}bius transform, defined as in \eqref{eqn: Mobius}. 

By the above, when $\Omega=\bn$, Theorem \ref{thm: key inequality} translates into the following.

\begin{thm}\label{thm: key inequality on bn}
Suppose $h$ is a holomorphic function defined on a neighborhood of the closed unit ball $\overline{\bn}$. Then there exists a constant $N$ such that for any function $f\in \Hol(\bn)$ and any $z,w\in\bn$,
$$
|h(z)f(w)|\lesssim \frac{|1-\langle w,z\rangle|^N}{(1-|w|^2)^{N+n+1}}\int_{E(w,1)}|h(\lambda)f(\lambda)|dv_n(\lambda),
$$
where $E(w,1)$ is the unit ball centered at $w$ with radius $1$ in the Bergman metric.
\end{thm}

Also recall inequality \eqref{eqn: temp 15} in the proof of Theorem \ref{thm: key inequality}. For $\Omega=\bn$, it translates into
\begin{equation}\label{eqn: temp 15 bn}
|f(w)|\lesssim\frac{1}{(1-|w|^2)^{N+n+1}}\int_{E(w,1)}|f(\lambda)h(\lambda)|\intd v_n(\lambda),\quad\forall w\in\bn.
\end{equation}

\begin{defn}\label{defn: Lamuh}
Let $h$ be a holomorphic function in an open neighborhood of $\clb$. Take the measure $d\mu_h=|h|^2dv_n$. Let $L^2(\mu_h)$ be the space of functions that are square integrable under this measure. Let $L_a^2(\mu_h)$ be the weighted Bergman space consisting of holomorphic functions in $L^2(\mu_h)$, that is,
\[
L_a^2(\mu_h)=\big\{f\in\Hol(\bn): \int_{\bn}|f|^2\intd\mu_h<\infty\big\}.
\]
\end{defn}
	
	\begin{lem}\label{mu_h complete}
		The weighted Bergman space $L_a^2(\mu_h)$ is a complete reproducing kernel Hilbert space.
	\end{lem}
	
	\begin{proof}
		By \eqref{eqn: temp 15 bn}, for any $f\in L_a^2(\mu_h)$ and $w\in\bn$,
		\begin{flalign*}
		|f(w)|\lesssim&(1-|w|^2)^{-(N+n+1)}\int_{E(w,1)}|fh|\intd v_n\\
		\leq&(1-|w|^2)^{-(N+n+1)}\int_{\bn}|fh|\intd v_n\\
		\lesssim&(1-|w|^2)^{-(N+n+1)}\bigg(\int_{\bn}|fh|^2\intd v_n\bigg)^{1/2}\\
		=&(1-|w|^2)^{-(N+n+1)}\|f\|_{L_a^2(\mu_h)}.
		\end{flalign*}
		Thus the evaluations functionals are bounded on $L_a^2(\mu_h)$. Given a compact subset $K\subset\bn$, the above also shows that the evaluation functionals at points in $K$ are uniformly bounded.
		
		Next we show that $L_a^2(\mu_h)$ is complete, or equivalently, $L_a^2(\mu_h)$ is closed in $L^2(\mu_h)$. Suppose $\{f_n\}\subseteq L_a^2(\mu_h)$ and $f_n$ converge to $f\in L^2(\mu_h)$. By the above, $f_n$ converge to $f$ uniformly on any compact subset of $\bn$.
	This shows that $f$ is holomorphic. Therefore $f\in L_a^2(\mu_h)$. So $L_a^2(\mu_h)$ is complete. This completes the proof.
	\end{proof}
	
	For $f\in\Hol(\bn)$ and $0<r<1$, write
	\[
	f_r(z)=f(rz), \forall z\in\cn, |z|<\frac{1}{r}.
	\]
	\begin{lem}\label{lem: poly dense in Lamuh}
		Suppose $h$ is a holomorphic function defined in an open neighborhood of $\clb$. Then there exists $C>0$ such that for any $f\in L_a^2(\mu_h)$ and $\frac{1}{2}<r<1$,
		\[
		\|f_r\|_{L_a^2(\mu_h)}\leq C\|f\|_{L_a^2(\mu_h)}.
		\]
		As a consequence, the set of holomorphic functions defined in a neighborhood of $\clb$ is dense in $L_a^2(\mu_h)$.
	\end{lem}
	
	\begin{proof}
		Applying Theorem \ref{thm: key inequality on bn} for $w=rz$, we get
		\begin{flalign*}
			|h(z)f(rz)|\lesssim&\frac{(1-r|z|^2)^N}{(1-r^2|z|^2)^{N+n+1}}\int_{E(rz,1)}|h(\lambda)f(\lambda)|\intd v_n(\lambda)\\
			\leq&\frac{1}{(1-r^2|z|^2)^{n+1}}\int_{E(rz,1)}|h(\lambda)f(\lambda)|\intd v_n(\lambda).
		\end{flalign*}
		Therefore
		\begin{flalign*}
			\|f_r\|^2_{L_a^2(\mu_h)}=&\int_{\bn}|h(z)f_r(z)|^2\intd v_n(z)\\
			\lesssim&\int_{\bn}\frac{1}{(1-r^2|z|^2)^{2(n+1)}}\bigg|\int_{E(rz,1)}|h(\lambda)f(\lambda)|\intd v_n(\lambda)\bigg|^2\intd v_n(z)\\
			\leq&\int_{\bn}\bigg(\frac{1}{(1-r^2|z|^2)^{2(n+1)}}v_n(E(rz,1))\cdot\int_{E(rz,1)}|h(\lambda)f(\lambda)|^2\intd v_n(\lambda)\bigg)\intd v_n(z)\\
			\lesssim&\int_{\bn}\bigg(\frac{1}{(1-r^2|z|^2)^{n+1}}\int_{E(rz,1)}|h(\lambda)f(\lambda)|^2\intd v_n(\lambda)\bigg)\intd v_n(z).
		\end{flalign*}
		By Fubini's Theorem, the last integral is equal to
		\begin{flalign*}
			&\int_{\bn}\bigg(\int_{\{z:~rz\in E(\lambda,1)\}}\frac{1}{(1-r^2|z|^2)^{n+1}}\intd v_n(z)\bigg)|h(\lambda)f(\lambda)|^2\intd v_n(\lambda)\\
			=&\int_{\bn}\bigg(\int_{E(\lambda,1)}\frac{1}{(1-|\eta|^2)^{n+1}}\frac{1}{r^{2n}}\intd v_n(\eta)\bigg)|h(\lambda)f(\lambda)|^2\intd v_n(\lambda)\\
			\lesssim&\int_{\bn}|h(\lambda)f(\lambda)|^2\intd v_n(\lambda)\\
			=&\|f\|^2_{L_a^2(\mu_h)}.
		\end{flalign*}
		Here we used the fact that $v_n(E(w,1))\approx(1-|w|^2)^{n+1}$, and $1-|\eta|^2\approx1-|\lambda|^2$ whenever $\eta\in E(\lambda,1)$ (cf. \cite{kehezhu}).
		
		We have proved the inequality. It remains to show that functions defined in a neighborhood of $\clb$ are dense. For any $f\in L_a^2(\mu_h)$, let $f_m:=f_{1-\frac{1}{m+1}}$. Then the sequence of functions $\{f_m\}$ are defined in a neighborhood of $\clb$. By the previous argument, this is a bounded sequence in $L_a^2(\mu_h)$. Therefore there exists a subsequence that converges weakly. Since $f_m\to f$ pointwise, the weak limit must be $f$. Thus $f$ lies in the weak closure of the subspace of function defined in a neighborhood of $\clb$. By the Hahn-Banach Theorem, $f$ also belongs to the norm closure. This completes the proof.
	\end{proof}

	\begin{proof}[{\bf Proof of Theorem \ref{thm: codim 1 on bn}}]
		Let $V$ be as in Theorem \ref{thm: codim 1 on bn}. Since $\bn$ is contractible, clearly $H^2(\bn,\mathbb{Z})=0$. Thus by Lemma \ref{lem: P_V for contractible}, there is a holomorphic function $h$ in an open neighborhood of $\clb$ such that
		\begin{equation}\label{eqn: temp P_V set}
		\PP_V=\big\{hf: f\in\Hol(\bn), hf\in L_a^2(\bn)\big\}=\big\{hf: f\in L_a^2(\mu_h)\big\}.
		\end{equation}
		The last equality follows easily from Definition \ref{defn: Lamuh}.
		Define the operator
		$$
		\mathcal{I}:L_a^2(\mu_h)\to L_a^2(\bn),\quad f\mapsto fh.
		$$
		Then $\mathcal{I}$ is an isometric isomorphism, and by the above, $\mathrm{Ran}(\mathcal{I})=\PP_V$. By Lemma \ref{lem: poly dense in Lamuh}, the subspace $\Hol(\clb)$ consisting of
		 functions that are holomorphic in a neighborhood of $\clb$ are dense in $L_a^2(\mu_h)$. Therefore 
		 \begin{equation}\label{eqn: P_V=P_h}
		 \PP_V=\mathrm{Ran}(\mathcal{I})=\mathcal{I}\big(L_a^2(\mu_h)\big)=\mathcal{I}\bigg(\overline{\Hol(\clb)}\bigg)=\overline{\mathcal{I}(\Hol(\clb))}=[h].
		 \end{equation}
	 Combining \eqref{eqn: temp P_V set} and \eqref{eqn: P_V=P_h} gives \eqref{eqn: P_V=P_h=set}. Finally, by Theorem \ref{thm: main}, the submodule $\PP_V=[h]$ is $p$-essentially normal for all $p>n$. It is well-known that for $p>n$, the $p$-essential normality of a submodule in $L_a^2(\bn)$ is equivalent to the corresponding quotient module. Thus $\PP_V^\perp=L_a^2(\bn)\ominus \PP_V$ is also $p$-essentially normal for $p>n$.
		This completes the proof.
	\end{proof}
	
	We end this paper with several remarks.
	\begin{rem}
	In Lemma \ref{lem: P_V for contractible}, we showed that if $\Omega$ satisfies the topological condition $H^2(\Omega,\mathbb{Z})=0$ then for a pure codimension $1$ variety $V$ in a neighborhood of $\overline{\Omega}$,
	\[
	\PP_V=\{hf: f\in\Hol(\Omega), hf\in L_a^2(\Omega)\}
	\]
	for some generating function $h$. The gap between $\PP_V$ and $[h]$ is then filled for $\Omega=\bn$. The technical difficulty of extending this result to any bounded strongly pseudoconvex domain $\Omega$ with smooth boundary, satisfying $H^2(\Omega,\mathbb{Z})=0$, lies in Lemma \ref{lem: poly dense in Lamuh}. We believe that this difficulty can be overcome by a detailed estimate (perhaps using the Narasimhan lemma \cite{krantz}).
	\end{rem}

	\begin{rem}
	Proving equalities like $\PP_V=[h]$ in Theorem \ref{thm: codim 1 on bn} allows us to study the Geometric Arveson-Douglas Conjecture by studying submodules. There is, however, a drawback in this approach: for $n\geq p>\dim_{\C}V$, the $p$-essential normality of a quotient module is no longer equivalent to that of its corresponding submodule. In fact a general submodule is not expected to be $p$-essentially normal for $p\leq n$. Nonetheless, the essential normality (the cross commutator being compact) of a quotient module in $L_a^2(\Omega)$ has already some important applications. For example, by the BDF-theory, an essentially normal quotient module $\mathcal{Q}$ defines an element in a K-homology $K_1(X)$, where $X$ a topology space depending on $\mathcal{Q}$. For nice varieties $V$, the $X$ space for $\mathcal{Q}_V=\PP_V^\perp$ is $V\cap\partial\Omega$. As observed by the first author in \cite{index}, this gives a new kind of index theory.
\end{rem}
	
	\noindent\textbf{Acknowledgement} The second author is partially supported by National Natural Science Foundation of China.
	The third author is partially supported by National Science Foundation(DMS-2101370).
	
	~ 
	
	A first draft of this paper had been completed in 2017. To our sadness, the first author, Ronald G. Douglas passed away in 2018. The second and third authors would like to write a few words in memory of the first author. Ronald G. Douglas was a master in the area of operator theory and operator algebra. Many of his work, including the BDF-theory, the Cowen-Douglas operator theory, the theory of Douglas algebras, and the theory of Hilbert Modules, have inspired generations of researchers. The conjecture studied in this paper also bears his name. The third author had the opportunity of working with him from 2014 to 2018. Discussing with Ron had always been enjoyable experiences. He sometimes talks like a philosopher. Some of his words remain influential today. Perhaps that is why it still feels unreal, even after four years.

		\bibliographystyle{plain}
	\bibliography{reference}

~

~

	\noindent Ronald G. Douglas, Department of Mathematics, Texas A\&M University, College Station, Texas, 77843, USA, E-mail: rdouglas@math.tamu.edu
	
	~
	
	\noindent Kunyu Guo, School of Mathematical Sciences, Fudan University, Shanghai, 200433, China, E-mail: kyguo@fudan.edu.cn
	
	~
	
	\noindent Yi Wang, Department of Mathematics, Vanderbilt University, Nashville, Tennessee, 37203, USA, E-mail: yiwangfdu@gmail.com
\end{document}